\documentclass[11pt,a4paper]{amsart}
\usepackage[margin=3cm]{geometry}
\usepackage{color}                    

\usepackage[pdftex,hyperfootnotes=false,colorlinks=true,linkcolor=blue,
citecolor=purple,filecolor=magenta,urlcolor=cyan,hypertexnames=false]{hyperref}

\usepackage{anyfontsize,enumitem,stmaryrd,thmtools,tikz,txfonts,nccmath}
\usepackage[oldstyle]{libertine}%
\usepackage[T1]{fontenc}
\usepackage[utf8]{inputenc}
\usepackage{csquotes}
\makeatletter
\renewcommand*\libertine@figurestyle{LF}
\makeatother
\usepackage[libertine,libaltvw,liby]{newtxmath}
\makeatletter
\renewcommand*\libertine@figurestyle{OsF}
\makeatother
\usepackage[style=alphabetic,maxalphanames=5,giveninits,sorting=nyt,%
maxbibnames=99]{biblatex}
\usepackage{tikz-cd}
\usepackage[noabbrev]{cleveref}
\expandafter\def\csname ver@etex.sty\endcsname{3000/12/31}

\usepackage{autonum}
\crefname{lemma}{lemma}{lemmata}
\Crefname{lemma}{Lemma}{Lemmata}
\crefname{subsection}{subsection}{subsections}
\Crefname{subsection}{Subsection}{Subsections}

\newtheorem{theorem}{Theorem}[section]
\newtheorem{lemma}[theorem]{Lemma}
\newtheorem{proposition}[theorem]{Proposition}

\theoremstyle{definition}
\newtheorem{definition}[theorem]{Definition}
\newtheorem{remark}[theorem]{Remark}

\newtheorem{notation}[theorem]{Notation}

\newcommand{\mb}[1]{\mathbb{#1}} 
\newcommand{\mf}[1]{\mathfrak{#1}}
\newcommand{\mc}[1]{\mathcal{#1}}

\newcommand{\N}{\mb{N}} 
\newcommand{\C}{\mb{C}} 
\newcommand{\Z}{\mb{Z}}

\newcommand{\<}{\langle}
\renewcommand{\>}{\rangle}

\DeclareMathOperator{\ad}{ad}
\newcommand{\bvs}[4]{\varsigma ' \left(\begin{smallmatrix} #1 & #2 \\ #3 & #4 \end{smallmatrix}\right)}

\newcommand{\dcor}[1]{\left\langle #1 \right\rangle }

\begingroup\expandafter\expandafter\expandafter\endgroup
\expandafter\ifx\csname pdfsuppresswarningpagegroup\endcsname\relax
\else
\fi
{}

\begin{document}
\title[Wall-crossing formulae and strong polynomiality for mixed Hurwitz numbers]{Wall-crossing formulae and strong piecewise polynomiality for mixed Grothendieck dessins d'enfant, monotone, and double simple Hurwitz numbers}
\author[M.~A.~Hahn]{Marvin Anas Hahn}
\address{M.~A.~H.: Mathematisches Institut, Universit\"at T\"ubingen, Auf der Morgenstelle 10, 72076 T\"ubingen, Germany.}
\email{marvinanashahn@gmail.com}
\author[R.~Kramer]{Reinier Kramer}
\address{R.~K.: Korteweg-de Vries Instituut voor Wiskunde, Universiteit van Amsterdam, P.O. Box 94248, 
1090 GE Amsterdam, Netherlands.}
\email{r.kramer@uva.nl}
\author[D.~Lewanski]{Danilo Lewanski}
\address{D.~L.: Max Planck Institut f\"{u}r Mathematik, Vivatsgasse 7, 53111 Bonn, Germany.}
\email{ilgrillodani@mpim-bonn.mpg.de}
\thanks{}

\begin{abstract}
We derive explicit formulae for the generating series of mixed Grothendieck dessins d'enfant/monotone/simple Hurwitz numbers, via the semi-infinite wedge formalism. This reveals the strong piecewise polynomiality in the sense of Goulden--Jackson--Vakil, generalising a result of Johnson, and provides a new explicit proof of the piecewise polynomiality of the mixed case. Moreover, we derive wall-crossing formulae for the mixed case. These statements specialise to any of the three types of Hurwitz numbers, and to the mixed case of any pair.
\end{abstract}
\maketitle
\tableofcontents

%%%%%%%%%%%%%%%%%%%%%%%%%%%%%%%%
%%%                                  SECTION
%%%%%%%%%%%%%%%%%%%%%%%%%%%%%%%%

\section{Introduction}
Hurwitz numbers have been introduced by Adolf Hurwitz in \cite{Hurwitz} as the count of  genus $g$, degree $d$ branched coverings of the Riemann sphere with given ramification profiles over a number of given fixed points. In the last two decades there have been many developments in the theory of Hurwitz numbers in different branches of mathematics and physics, including enumerative algebraic geometry, differential geometry, tropical geometry, combinatorics, representation theory,  integrable systems, and random matrix models. For a recent textbook on Hurwitz theory we refer to \cite{CM}.\par
Some Hurwitz numbers have proven to be of particular interest. Double simple Hurwitz numbers $h_{g, \mu, \nu}$ count coverings of genus $g$ and degree $d$ of the Riemann sphere with two fixed ramification profiles $\mu$ and $\nu$ over the points $0$ and $\infty$, and over the other $b$ fixed points on $\mathbb{P}^1$, the ramification profile must be simple. Hence $\mu, \nu$ are partitions of $d$ and, by the Riemann-Hurwitz formula, $b = 2g - 2 +\ell(\mu) + \ell(\nu)$.  There exist several modifications of the condition of simplicity on the intermediate ramifications, whose corresponding numbers also provide rich structures.\par
Two of those will be important in this paper. Labelling the sheets of the covering from $1$ to $d$, every intermediate simple ramification corresponds to a transposition $(a_i, b_i)_{i=1, \dots, b}$ that can be written such that $a_i < b_i$. The double monotone Hurwitz numbers are defined as the same count of the double simple Hurwitz numbers with the extra requirement that the coverings should satisfy the condition $b_i \leq b_{i+1}$. The strictly monotone Hurwitz case requires that $b_i < b_{i+1}$. For each of these three definitions, specialising $\nu$ to the trivial partition $(1^d)$ one obtains the single version of the corresponding double Hurwitz number.

\subsubsection*{Single simple Hurwitz numbers}

It was observed in \cite{Goulden1997} that the Hurwitz numbers in genus zero exhibit polynomiality in the entries of the partition $\mu$, up to a combinatorial prefactor. The generalisation in any genus could later be derived from the celebrated  Ekedahl--Lando--Shapiro--Vainshtein (ELSV) formula \cite{ekedahl2001hurwitz}, that expresses single simple Hurwitz numbers in terms of the intersection theory of the Deligne--Mumford compactification of the moduli spaces of curves. It has moreover been proved that the single simple Hurwitz numbers satisfy the topological recursion \cite{EMS} in the sense of Chekhov--Eynard--Orantin (CEO) \cite{EO}. It is a general fact \cite{DOSS, E} that numbers satisfying CEO recursion admit an expression in terms of the intersection theory of the moduli space of curves, although this expression may be hard to derive.

\subsubsection*{Double simple Hurwitz numbers}
In \cite{Okounkov}, Okounkov proved that the double simple Hurwitz numbers can be expressed in terms of the semi-infinite wedge formalism by exhibiting explicit operators. This rephrasing implies a relation to integrable systems of Kadomtsev-Petviashvili (KP) type---more precisely, the partition function of the double Hurwitz numbers is a tau-function of the KP integrable hierarchy.\par
A combinatorial approach to the double simple Hurwitz numbers appears in the foundational paper of Goulden, Jackson, and Vakil  \cite{GJV}, in which it is proved that the double simple Hurwitz numbers are piecewise polynomial in the entries of $\mu$ and $\nu$. Roughly speaking, relative conditions on $\mu$ and $\nu$ determine hyperplanes (walls) in the configuration space of these partitions. The complement of the walls is divided in several distinct connected components, which are called chambers. The piecewise polynomiality property means that, inside each chamber, there exist a polynomial depending on the chamber whose evaluations at the entries of $\mu$ and $\nu$ coincide with the Hurwitz numbers under examination. Moreover, in the same paper they proposed a conjecture of strong piecewise polynomiality,  proposing a lower bound on the degree of the polynomial. This lower bound is considered an indication of the connection with intersection theory of moduli spaces, as it shows up as consequence of the ELSV formula in the case of single simple Hurwitz numbers.\par
The chamber structure and wall-crossing formulae in genus zero for double Hurwitz numbers have been studied with algebro-geometric methods by Shadrin, Shapiro, and Vainshtein \cite{SSV}. A tropical approach to double Hurwitz numbers has been developed by Cavalieri, Johnson, and Markwig \cite{CJM}. This approach led the same authors to determine the chamber structure and wall crossing formulae in any genera \cite{CJMa}. Finally, the strong piecewise polynomiality conjecture has been proved by Johnson in \cite{johnson2015}. He used the operator language of \cite{Okounkov} to derive an explicit algorithm to compute the chamber polynomials and the wall-crossing formulae. A precise conjecture concerning CEO topological recursion for double Hurwitz numbers appears in \cite{NuovoNormanMax}, whereas an ELSV formula for double Hurwitz numbers remains an open problem.

 \subsubsection*{Monotone Hurwitz numbers}
The monotone Hurwitz numbers have been introduced in \cite{GGN14} as a combinatorial interpretation of the asymptotic expansion of the Harish-Chandra--Itzykson--Zuber (HCIZ) random matrix model. The CEO topological recursion for the simple case was proved in \cite{do2014topological} and its generalisation to the orbifold case appears in \cite{DK}. An ELSV formula for the simple case is derived in \cite{ALS} and \cite{DK}, whereas an ELSV formula for the double case is still an active topic of research. A tropical approach for the monotone case is developed in \cite{DK} and in \cite{Marvinmixed}.

\subsubsection*{Grothendieck dessins d'enfant or strictly monotone Hurwitz numbers}
Dessins d'enfant have been introduced by Grothendieck in \cite{Grothendieck}. Their enumeration counts Hurwitz coverings of genus $g$ and degree $d$ over the Riemann sphere, with two ramifications $\mu$ and $\nu$ over $0$ and $\infty$, and a single further ramification over $1$, whose length is determined by the Riemann--Hurwitz formula.\par
The CEO recursion for the $r=2$ orbifold case (i.e. $\nu = (2)^{d/2}$) was proved in \cite{PaulRibbon,DMSS} and is known as enumeration of ribbon graphs. In \cite{DoManescu} the CEO recursion was conjectured for the general $r$-orbifold case (i.e. $\nu = (r)^{d/r}$), which was then proved in \cite{DOPS} by combinatorial methods. Moreover, CEO recursion was proved in \cite{KazarianZograf} for the case of two consecutive intermediate ramifications of fixed lengths instead of one, together with a proof of the KP integrability and the Virasoro constraints for the same case. ELSV formulae for simple, orbifold or double cases are still not known. The connection between strictly monotone numbers and dessins d'enfant counting is explained in the following. 

\subsubsection*{Mixed cases} It is natural to interpolate several Hurwitz enumerative problems, by allowing different conditions on different blocks of intermediate ramifications. 
In fact hypergeometric tau functions for the 2D Toda integrable hierarchy have been proved to have several explicit combinatorial interpretations \cite{HO2}---one of them is in terms of mixed double strictly monotone/weakly monotone Hurwitz numbers, another one involves a mixed case of combinatorial problems, in which the part relative to the strictly monotone ramifications can be interpreted in terms of Grothendieck dessins d'enfant. This implies indirectly that the enumeration of Grothendieck dessins and strictly monotone numbers coincide. A direct proof of this fact through the Jucys correspondence \cite{Jucys} is derived in \cite{ALS}.\par
A combinatorial study of the mixed double monotone--simple case can be found in \cite{zbMATH06586291}, in which piecewise polynomiality is proved. %In the same paper it is also observed that .
A tropical interpretation providing an algorithm to compute the chamber polynomials and wall-crossing formulae via Erhart theory is developed in \cite{Marvinmixed}. Further developments on CEO topological recursion for general Hurwitz enumerative geometric problems appear in \cite{alexandrov2016weighted}. This study confirms the existence of an ELSV-type formula for mixed Hurwitz enumerative problems.

 \subsection{Results}
We derive explicit formulae for the generating functions of mixed double Grothendieck/\allowbreak{}monotone/\allowbreak{}simple  Hurwitz numbers. As a corollary, this provides the strong polynomiality statement for the mixed monotone/simple case (generalising a result of \cite{johnson2015}), and furthermore its generalisation to the mixed monotone/Grothendieck/simple case. In particular, this provides a new explicit proof of the piecewise polynomiality of the mixed case, and the obtained expressions allow us to derive wall-crossing formulae. These results specialise to the three types of Hurwitz numbers and to the mixed case of any pair, hence in particular this generalises the wall-crossing formulae derived in \cite{Marvinmixed}.

Our methods rely on the application of the algorithm introduced by Johnson in \cite{johnson2015}, that we taylor slightly for our use. The new key ingredients to run the algorithm in this case are the operators for the monotone and the strictly monotone ramifications, derived in \cite{ALS}.

\subsection{Organisation of the paper}
In \cref{sec:preliminaries}, we define the Hurwitz numbers used and give their operator representation in the semi-infinite wedge formalism. In section \ref{sec:Johnson} we recall Johnson's algorithm and adapt it to our purpose, deriving our first main result, \cref{theorem:Joh_alg}. The main structural results of the paper are described in the last three sections. In section \ref{sec:piecewise} we apply \cref{theorem:Joh_alg} to obtain piecewise polynomiality results. \Cref{piecewisemonotone} deals with the cases of double monotone and double Grothendieck's Hurwitz numbers, and \cref{lowestdegree} checks whether the lowest degree in the polynomials is really non-zero. In \cref{piecewisemixed} we treat the mixed monotone/Grothendieck/simple case, and we derive the strong piecewise polynomiality statement. Section \ref{sec:wall} is devoted to the derivation of the wall-crossing formulae. Finally, \cref{HypGeomTau} gives a polynomiality result for hypergeometric tau functions of the 2D-Toda hierarchy.

\subsection{Acknowledgements} 
We are indebted to  H.~Markwig for her careful proofreading and many helpful suggestions on an early version of this paper. We thank G.~Borot, N.~Do, E.~Garcia-Failde, M.~Karev, A.~Popolitov, and S.~Shadrin for useful discussions, and the Institut Henri Poincar\'{e} in Paris that hosted the trimester ``\textit{Combinatorics and interactions}'', where this collaboration started.
 R. K. and D. L. want to thank S.~Shadrin for introducting them to the semi-infinite wedge formalism. M. A. H. gratefully acknowledges partial support by DFG SFB-TRR 195 ``Symbolic tool in mathematics and their applications'', project A 14 ``Random matrices and Hurwitz numbers'' (INST 248/238-1). R. K. and D. L. were supported by a VICI grant of the Netherlands Organization for Scientific Research.

%%%%%%%%%%%%%%%%%%%%%%%%%%%%%%%%
%%%                                  SECTION
%%%%%%%%%%%%%%%%%%%%%%%%%%%%%%%%

\section{Hurwitz numbers and the semi-infinite wedge formalism}\label{sec:preliminaries}

In this section, we recall the basic notions required for our work, for which we need the following conventions. We write \( \Z' \coloneqq \Z + \frac{1}{2} \). For partitions \( \mu, \nu \), we set \( m \coloneqq \ell (\mu ) \), and \( n \coloneqq \ell (\nu )\). We also define the functions \( \varsigma (z) \coloneq e^{z/2} -e^{-z/2} \) and \( \mc{S} (z) \coloneq \frac{\varsigma (z)}{z} \).

\subsection{Triply mixed Hurwitz numbers}
Initially, Hurwitz numbers were introduced as topological invariant counting ramified coverings between Riemann spheres.
\begin{definition}[Double simple Hurwitz numbers]
\label{def:hurwitz}
 Let $d$ be a positive integer, $\mu,\nu$ two ordered partitions of $d$ and let $g$ be a non-negative integer. Moreover, let $q_1,\dots,q_b$ be distinct points in $\mathbb{P}^1$, where $b=2g-2+m+n$. We define a \emph{simple Hurwitz cover} of type $(g,\mu,\nu)$ to be a map $\pi:C\to\mathbb{P}^1$, such that:
  \begin{enumerate}
  \item $C$ is a (not necessarily connected) genus $g$ curve;
  \item $\pi$ is a degree $d$ map, with ramification profile $\mu$ over $0$, $\nu$ over $\infty$, and $(2,1,\dots,1)$ over $q_i$ for all $i=1,\dots,b$;
  \item $\pi$ is unramified everywhere else;
  \item the pre-images of $0$ and $\infty$ are labeled, such that the point labeled $i$ in $\pi^{-1}(0)$ (respectively $\pi^{-1}(\infty)$) has ramification index $\mu_i$ (respectively $\nu_i)$.  
 \end{enumerate}
 A simple Hurwitz cover is \emph{connected} if its domain is.\par
We define an isomorphism between two covers $\pi_1:C_1\to\mathbb{P}^1$ and $\pi_2:C_2\to\mathbb{P}^1$ to be a morphism $\varphi:C_1\to C_2$ respecting the labels, such that the following diagram commutes:
\begin{equation}
\begin{tikzcd}
C_1 \arrow[r, "\varphi"] \arrow[d, "\pi_1"]
& C_2 \arrow[d, "\pi_2"] \\
\mathbb{P}^1 \arrow[r, "\mathrm{id}"]
& \mathbb{P}^1
\end{tikzcd}\,.
\end{equation}
Then we define the \emph{(not neccessarily connected) double simple Hurwitz numbers} as follows: 
\begin{equation}
h_{g;\mu,\nu}^\bullet =\sum \frac{1}{|\mathrm{Aut}(\pi)|}\,,
\end{equation}
where the sum goes over all isomorphism classes of Hurwitz covers of type $(g,\mu,\nu)$. This number does not depend on the position of the $q_i$. The degree is implicit in the notation $h_{g;\mu,\nu}$, as $d=\sum \mu_i=\sum \nu_j$. The number $b$ of simple branch points is determined by the Riemann-Hurwitz formula, so $b=2g-2+m+n$ as above.\par
We define the \emph{connected double simple Hurwitz numbers} \( h^\circ_{g;\mu,\nu}\) in the same way, but summing only over connected Hurwitz covers.
\end{definition}
Generically, the numbers \( h^\bullet_{g;\mu,\nu} \) and \( h^\circ_{g;\mu,\nu}\) agree: disconnected covers only exist if there are non-trivial subpartitions of \( \mu \) and \( \nu \) of equal size, corresponding to the ramification profiles of one of the connected components. This condition defines a number of codimension one subspaces in the space of all pairs of partitions, see \cref{hyparr}. As most of this paper considers Hurwitz numbers outside this subspace, we will often neglect mentioning whether we consider connected or disconnected Hurwitz numbers.\par
\vspace{11pt}
Hurwitz numbers can also be defined via decompositions of the identity in the symmetric group. For $\sigma\in S_d$, we denote its cycle type by $\mathcal{C}(\sigma)\vdash d$. We define the following factorisation counting problem in the symmetric group:
\begin{definition}[Factorisations in the symmetric group]
\label{def:hursym}
Let $d,g,\mu,\nu$ be as in \cref{def:hurwitz}. We call $\left(\sigma_1,\tau_1,\dots,\tau_b,\sigma_2\right)$ a \emph{factorisation of type $(g,\mu,\nu)$} if:
\begin{enumerate}
\item $\sigma_1,\ \sigma_2,\ \tau_i\in\mathcal{S}_d$;
\item $\sigma_2\cdot\tau_b\cdot\dots\cdot\tau_1\cdot\sigma_1=\mathrm{id}$;
\item $b=2g-2+m+n$;
\item $\mathcal{C}(\sigma_1)=\mu,\ \mathcal{C}(\sigma_2)=\nu$ and $\mathcal{C}(\tau_i)=(2,1,\dots,1)$;
%\item the group generated by $\left(\sigma_1,\tau_1,\dots,\tau_b,\sigma_2\right)$ acts transitively on $\{1,\dots,d\}$;
\item the disjoint cycles of $\sigma_1$ and $\sigma_2$ are labeled, such that the cycle $i$ has length $\mu_i$.
\end{enumerate}
We denote the set of all factorisations of type $(g,\mu,\nu)$ by $\mathcal{F}(g,\mu,\nu)$.
\end{definition}
A well-known fact is the following theorem, which is essentially due to Hurwitz.
\begin{theorem}
Let $g,\mu,\nu$ and $h_{g;\mu,\nu}$ and $\mathcal{F}(g,\mu,\nu)$ be as in the previous definition. Then
\begin{equation}
h _{g;\mu,\nu}=\frac{1}{d!}\left|\mathcal{F} (g,\mu,\nu)\right|\,.
\end{equation}
\end{theorem}

\begin{remark}
The connected double simple Hurwitz numbers $h^{\circ}_{g;\mu,\nu}$ may be computed in the symmetric group as well. This is accomplished by adding to \cref{def:hursym} the condition that the subgroup of $\mathcal{S}_d$ generated by $\sigma_1,\sigma_2,\tau_1,\dots,\tau_b$ acts transitively on the set $\{1,\dots,d\}$.
\end{remark}

%As proved in \cite{GGN14} double monotone Hurwitz numbers appear as the coefficients of the HCIZ-integral. As mentioned in the introduction, it is natural to study interpolations between different Hurwitz-type problems

The following definition is a natural generalisation of the notion of mixed Hurwitz numbers studied in \cite{zbMATH06586291}.

\begin{definition}[Triply mixed Hurwitz numbers]
Let $g$, $p$, $q$, $r$ be non-negative integers and let $\mu$ and $\nu$ be ordered partitions, such that $b \coloneq p+q+r=2g-2+m+n$. We call a tuple $(\sigma_1,\tau_1,\dots,\tau_b,\sigma_2)$, a \textit{triply mixed factorisation of type } $(g,\mu,\nu,p,q,r)$ if it is a factorisation of type $(g,\mu,\nu)$ and for $\tau_i=(r_i\;s_i)$, where $r_i>s_i$ we have
\begin{enumerate}
\item [(6)] $s_{i+1}\ge s_{i}$ for $i=p+1,\dots,p+q$,
\item [(7)] $s_{i+1}> s_{i}$ for $i=p+q+1,\dots,b$.
\end{enumerate}
We denote the set of all triply mixed factorisations of type $(g,\mu,\nu,p,q,r)$ by $\mathcal{F}_{p,q,r;\mu,\nu}^{(2),\le,<}$ and we define the \emph{triply mixed Hurwitz numbers}
\begin{align}
h_{p,q,r;\mu,\nu}^{(2),\le,<}\coloneq \frac{1}{d!}\left|\mathcal{F}_{p,q,r;\mu,\nu}^{(2),\le,<}\right|\,.
\end{align}
\end{definition}

\begin{remark}
Triply mixed Hurwitz numbers can be thought of as a two-dimensional combinatorial interpolation between different Hurwitz-type counts:
\begin{enumerate}
\item For $q=r=0$, we obtain the double simple Hurwitz numbers.
\item For $p=q=0$, we obtain the double strictly monotone Hurwitz numbers, denoted by $h^{<}_{g;\mu,\nu}$. In \cite{ALS}, it was proved that this number is equivalent to the Grothendieck dessins d'enfant count as explained in the introduction.
\item For $p=r=0$, we obtain the double monotone Hurwitz numbers, denoted by $h^{\le}_{g;\mu,\nu}$.
\end{enumerate}
Triply mixed Hurwitz numbers are a generalisation of the notion of mixed double Hurwitz numbers introduced in \cite{zbMATH06586291}, which corresponds to the one-dimensional interpolation between double simple and monotone Hurwitz numbers, i.e. $r=0$.
\end{remark}

It is natural to ask whether Hurwitz-type counts behave polynomially in some sense. In particular, we define the subspace 
\begin{equation}
\mathcal{H}(m,n)=\Big\{ (\underline{M},\underline{N})\, \Big| \, \underline{M} \in \mb{N}^m,\underline{N} \in \mb{N}^n \textrm{, such that} \medop\sum_{i=1}^m M_i = \medop\sum_{j=1}^n N_j \Big\} \subset \mb{N}^m \times \mb{N}^n \,,
\end{equation}
where $\underline{M} = (M_1, \dots, M_m)$ and $\underline{N} = (N_1, \dots, N_n)$ and view triply mixed Hurwitz numbers as a function in the following sense
\begin{equation}
h_{p,q,r}^{(2),\le,<}\colon \mathcal{H}(m,n)\to\mathbb{Q} \colon (\mu,\nu)\mapsto h_{p,q,r;\mu,\nu}^{(2),\le,<}\,.
\end{equation}
\begin{definition}\label{hyparr}
We define the \emph{hyperplane arrangement} $\mathcal{W}(m,n)\subset\mathcal{H}(m,n)$ induced by the family of linear equations $\sum_{i\in I}M_i=\sum_{j\in J}N_j$ for $I\subset[m]$, $J\subset[n]$, where the variables $M_i$ correspond to $\mathbb{N}^{m}$ and the variables $N_j$ correspond to $\mathbb{N}^{n}$. We call the hyperplanes induced by each equation \textit{the walls of the hyperplane arrangement} and the  sets of all \( (\underline{M},\underline{N})\) at the same side of each wall \textit{the chambers of the hyperplane arrangement}.
\end{definition}
Recall that in the chambers of the hyperplane arrangement, the connected and disconnected Hurwitz numbers agree.\par
In \cite{zbMATH06586291}, the following theorem was proved:

\begin{theorem}[\cite{zbMATH06586291}]
\label{theorem:mixhur}
Let $g$, $p$, $q$ be non-negative integers and let $\mu$ and $\nu$ be partitions such that $p+q=2g-2+m+n$. Then for each chamber $\mf{c}$ of $\mathcal{W}(m,n)$ there exists a polynomial $P_{p,q,0;\mu,\nu}^{(2),\le,<}\in\mathbb{Q}[M_i,N_j]$, such that\begin{align}
h_{p,q,0;\mu,\nu}^{(2),\le,<}=P_{p,q,0;\mu,\nu}^{(2),\le,<}
\end{align}
for all $(\mu,\nu)\in \mf{c}$.
\end{theorem}

\subsection{Semi-infinite wedge formalism and the operators for Hurwitz numbers}\label{WedgeAndOperators}
We introduce the operators needed for the derivation of our results. For a self-contained introduction to the infinite wedge space formalism, we refer the reader to \cite{OP,johnson2015}, where most relevant objects are defined. 
\par
Let $V = \bigoplus_{i \in \Z'} \C \underline{i}$ be an infinite-dimensional complex vector space with a basis labeled by half-integers, written as \( \underline{i}\). The semi-infinite wedge space \( \mc{V} \coloneq \bigwedge^{\frac{\infty}{2}} V \) is the space spanned by vectors
\begin{equation}
\underline{k_1} \wedge \underline{k_2} \wedge \underline{k_3} \wedge \cdots
\end{equation}
such that for large \( i\), \( k_i + i -\frac{1}{2} \) equals a constant, called the charge, imposing that \( \wedge \) is antisymmetric. The charge-zero sector 
\begin{equation}
\mathcal{V}_0 = \bigoplus_{n \in \mathbb{N} } \bigoplus_{\lambda\,  \vdash\, n} \C v_{\lambda}
\end{equation}
is then the span of all of the semi-infinite wedge products \( v_{\lambda} = \underline{\lambda_1 - \frac{1}{2}} \wedge \underline{\lambda_2 - \frac{3}{2}} \wedge \cdots \) for integer partitions~$\lambda$. The space $\mathcal{V}_0$ has a natural inner product $(\cdot,\cdot )$ defined by declaring its basis elements to be orthonormal. 
The element corresponding to the empty partition $v_{\emptyset}$ is called the vacuum vector and denoted by $|0\rangle$. Similarly, we call the covacuum vector its dual in \( \mc{V}_0^* \), and denote it by $\langle0|$. If $\mathcal{P}$ is an operator acting on $\mathcal{V}_0$, we denote with $ \langle \mathcal{P}\rangle$ the evaluation $ \langle 0 |  \mathcal{P} |0\rangle$.

For $k$ half-integer, define the operator $\psi_k$ by $\psi_k : (\underline{i_1} \wedge \underline{i_2} \wedge \cdots) \ \mapsto \ (\underline{k} \wedge \underline{i_1} \wedge \underline{i_2} \wedge \cdots)$, and let $\psi_k^{\dagger}$ be its adjoint operator with respect to~$(\cdot,\cdot)$. The normally ordered products of $\psi$-operators
\begin{equation}
E_{i,j} \coloneqq \begin{cases}\psi_i \psi_j^{\dagger}, & \text{ if } j > 0 \\
-\psi_j^{\dagger} \psi_i & \text{ if } j < 0 \end{cases} 
\end{equation}
are well-defined operators on $\mathcal{V}_0$. % We will use in the following definitions and computations the functions
%\begin{equation}
%\varsigma (z)=e^{z/2} - e^{-z/2} = 2 \sinh(z/2), \qquad \qquad 
%\mathcal{S}(z) = \frac{\varsigma (z)}{z}.
%\end{equation} 
For $n$ any integer, and $z$ a formal variable, define the operators
\begin{equation}
\mathcal{E}_n(z) = \sum_{k \in \Z + \frac12} e^{z(k - \frac{n}{2})} E_{k-n,k} + \frac{\delta_{n,0}}{\varsigma(z)}, \qquad \qquad \alpha_n = \mathcal{E}_n(0) = \sum_{k \in \Z + \frac12} E_{k-n,k}.
\end{equation}
Their commutation formulae are known to be
\begin{equation}
  \left[\mathcal{E}_a(z),\mathcal{E}_b(w)\right] =
\varsigma\left(\det  \left[
\begin{matrix}
  a & z \\
b & w
\end{matrix}\right]\right)
\,
\mathcal{E}_{a+b}(z+w), 
\qquad \qquad 
[\alpha_k, \alpha_l ] = k \delta_{k+l,0}.
\end{equation}
 We will also use the $\mathcal{E}$-operator without the correction in energy zero, i.e.
\begin{equation}
\tilde{\mathcal{E}}_0(z) = \sum_{k \in \Z + \frac12} e^{zk} E_{k,k} = \sum_{r=0}^{\infty} \mathcal{F}_r z^r, 
\qquad \qquad
 \mathcal{F}_r \coloneqq \sum_{k\in\Z+\frac12} \frac{k^r}{r!} E_{k,k} 
\end{equation}
The operator \( C = \mc{F}_0 \) is called the \emph{charge operator}, as its eigenvalues on basis vectors are given by the charge. In particular, it acts as zero on \( \mc{V}_0 \). The operator $E = \mathcal{F}_1$ is called the \emph{energy operator}. We are now ready to express the Hurwitz numbers in terms of the semi-infinite wedge formalism.\par
The monotone Hurwitz numbers have the following expression, derived in \cite{ALS}: 
\begin{equation}
h_{g;\mu,\nu}^{\le}=\frac{[u^b]}{\prod\mu_i\prod\nu_j}\left\langle \prod_{i=1}^m \alpha_{\mu_i} \mathcal{D}^{(h)}(u) \prod_{j=1}^n\alpha_{-\nu_j}\right\rangle\,,
\end{equation}
where the operator $\mathcal{D}^{(h)}(u)$ has the vectors $v_{\lambda}$ as eigenvectors with the generating series for the complete homogeneous polynomials $h$ evaluated at the content $\mathbf{cr}^{\lambda}$ of the Young tableau $\lambda$ as eigenvalues:
\begin{equation}
\mathcal{D}^{(h)}(u).v_{\lambda} = \sum_{v=0} h_v(\mathbf{cr}^{\lambda}) u^v v_{\lambda}\,.
\end{equation}
Remember that the content of a box \( (i,j) \) in the Young tableau of a partition is given by \( \mathbf{cr}_{(i,j)} = j-i \), and the content \( \mathbf{cr}^\lambda \) of a partition \( \lambda \) is the multiset of all contents of boxes in its Young diagram (\( \mathbf{cr} \) stands for \textbf{c}olumn-\textbf{r}ow). For example, the partition \( (3,2) \) has boxes \( (1,1)\), \( (1,2)\), \( (1,3)\), \((2,1)\), and \( (2,2) \), so \( \mathbf{cr}^{(3,2)} = \{ 0,1,2,-1,0\} \).\par 
Explicitly, the operator $\mathcal{D}^{(h)}$ can be expressed in terms of the operators $\mathcal{E}$ as
\begin{equation}
\mathcal{D}^{(h)}(u) = \exp\left( \left[ \frac{\tilde{\mathcal{E}}_0\left(u^2 \frac{d}{du}\right)}{\varsigma \left(u^2 \frac{d}{du}\right)} - E \right].\log(u) \right)\,.
\end{equation}
 Let $\mathcal{O}^{(h)}_{x}(u)$ indicate the conjugation $\mathcal{D}^{(h)}(u)\alpha_{x}\mathcal{D}^{(h)}(u)^{-1}$. Since $\mathcal{D}^{(h)}(u)^{-1}.|0\rangle = |0\rangle$, one can insert the operator $\mathcal{D}^{(h)}(u)^{-1}$ on the right and insert $1 = \mathcal{D}^{(h)}(u)\mathcal{D}^{(h)}(u)^{-1}$ between every consecutive pair of operators $\alpha_{-\nu_i}$, obtaining
\begin{equation}
h_{g;\mu,\nu}^{\le}=\frac{[u^b]}{\prod\mu_i\prod\nu_j}\left\langle \prod_{i=1}^m\mathcal{E}_{\mu_i}(0)\prod_{j=1}^n\mathcal{O}^{(h)}_{-\nu_j}(u)\right\rangle\,.
\end{equation}
The operators $\mathcal{O}^{(h)}$ have been computed in \cite{kramer2016quasi} to be equal to  
\begin{equation}
\mathcal{O}^{(h)}_{-\nu}(u)=\sum_{v=0}^{\infty}\frac{(v+\nu-1)!}{(\nu-1)!}[z^v]\mathcal{S}(uz)^{\nu-1}\mathcal{E}_{-\nu}(uz)\,,\label{eq:h_operators}
\end{equation}
so the result is the following lemma, which is the first key observation for this paper.

\begin{lemma}
\label{eq:h_as_wedge}
Let $g$ be a non-negative number, $\mu$ and $\nu$ partitions of the same positive integer. The monotone Hurwitz number corresponding to these data can be computed as
\begin{equation}
h_{g;\mu,\nu}^{\le} \! = \frac{[u^b]}{\prod \mu_i} \!\! \sum_{\substack{v \vdash b \\ \ell(v)=n}} \prod_{j=1}^n\frac{(v_j\! +\! \nu_j \!- \! 1)!}{\nu_j!}[z^{v_1}_1\dots z^{v_n}_n]\prod_{j=1}^n\mathcal{S}(uz_j)^{\nu_j-1}\left\langle\prod_{i=1}^m\mathcal{E}_{\mu_i}(0)\prod_{j=1}^n\mathcal{E}_{-\nu_j}(uz_j)\right\rangle\,.
\end{equation}
\end{lemma}
Similarly, the strictly monotone Hurwitz numbers can be expressed in the same way, substituting for $\mathcal{D}^{(h)}$ the operator $\mathcal{D}^{(\sigma)}(u) := \mathcal{D}^{(h)}(-u)^{-1}$. This reads
\begin{equation}
h_{g;\mu,\nu}^{<}=\frac{[u^b]}{\prod\mu_i\prod\nu_j}\left\langle \prod_{i=1}^m\mathcal{E}_{\mu_i}(0)\prod_{j=1}^n\mathcal{O}^{(\sigma)}_{-\nu_j}(u)\right\rangle\,,
\end{equation}
where
\begin{equation}
\mathcal{O}^{(\sigma)}_{-\nu}(u)=\sum_{v=0}^{\nu}\frac{\nu!}{(\nu-v)!}[z^v]\mathcal{S}(uz)^{-\nu-1}\mathcal{E}_{-\nu}(uz)\,,\label{eq:sigma_operators}
\end{equation}
and we obtain the following lemma in a fashion analogous to \cref{eq:h_as_wedge}
\begin{lemma}
\label{eq:sigma_as_wedge}
Let $g$ be a non-negative number, $\mu$ and $\nu$ partitions of the same positive integer. The strictly monotone Hurwitz number corresponding to these data can be computed as
\begin{equation}
 h_{g;\mu,\nu}^{<}=\frac{[u^b]}{\prod\mu_i}\!\! \sum_{\substack{v \vdash b\\ 0 \leq v_j \leq \nu_j }}\prod_{i=1}^n\frac{(\nu_j-1)!}{(\nu_j-v_j)!}[z^{v_1}_1\dots z^{v_n}_n]\prod_{j=1}^n\mathcal{S}(uz_j)^{-\nu_j-1}\left\langle\prod_{i=1}^m\mathcal{E}_{\mu_i}(0)\prod_{j=1}^n\mathcal{E}_{-\nu_j}(uz_j)\right\rangle\,.
\end{equation}
\end{lemma}

\begin{remark}
As seen in \cite{OP}, the double simple Hurwitz number can be computed as \begin{equation}
h_{g;\mu,\nu}=\frac{[u^b]}{\prod\mu_i\prod\nu_j}\left\langle\prod_{i=1}^m\mathcal{E}_{\mu_i}(0)\prod_{j=1}^n\mathcal{E}_{-\nu_j}(u\nu_j)\right\rangle.
\end{equation}
We see that the double monotone and strictly monotone Hurwitz numbers are computed as linear combinations of vacuum expectations similar to the ones appearing in the equation for double simple Hurwitz numbers.
\end{remark}

\subsection{The triply mixed Hurwitz case}

In order to apply the semi-infinite wedge formalism and Johnson's algorithm to the triply mixed case, it is best to consider a generating function. By the previous discussion, we can express such a generation function for the triply mixed Hurwitz numbers as follows:
\begin{equation}
\sum_{p,q,r =0}^\infty h_{p,q,r;\mu, \nu}^{(2),\leq,<} \frac{X^p}{p!} Y^q Z^r = 
 \frac{1}{\prod\mu_i\prod\nu_j } \left\langle\prod_{i=1}^m\mathcal{E}_{\mu_i}(0) e^{X \mathcal{F}_2} \mathcal{D}^{(h)}(Y) \mathcal{D}^{(\sigma)}(Z)\prod_{j=1}^n\mathcal{E}_{-\nu_j}(0)\right\rangle\,.
\end{equation}
In the same way as in \cref{WedgeAndOperators}, we can rearrange the correlator as
 \begin{equation}\label{TriplyMixedGeneratingConj}
\frac{1}{\prod\mu_i\prod\nu_j } \left\langle\prod_{i=1}^m\mathcal{E}_{\mu_i}(0) \prod_{j=1}^n e^{X \mathcal{F}_2} \mathcal{D}^{(h)}(Y) \mathcal{D}^{(\sigma)}(Z)\mathcal{E}_{-\nu_j}(0) \mathcal{D}^{(\sigma)}(Z)^{-1} \mathcal{D}^{(h)}(Y)^{-1} e^{-X \mathcal{F}_2}\right\rangle\,.
\end{equation}

In order to use this expression, we should calculate the conjugations in this correlator. this we do in the following two lemmata.

\begin{lemma}\label{lem:ConjF2} 
The conjugation with the exponential of $\mathcal{F}_2$ acts on the operator $\mathcal{E}_{-\nu}(z)$ by shifting the variable $z$ by the opposite of the energy. Explicitly:
\begin{equation}
e^{u \mathcal{F}_2} \mathcal{E}_{-\nu}(A) e^{-u \mathcal{F}_2} = \mathcal{E}_{-\nu}(A + u\nu)\,.
\end{equation}
\end{lemma}
\begin{proof}
This is a fairly straightforward computation. By a standard Lie theory result, the left-hand side is equal to 
 \begin{align}
e^{u \mathcal{F}_2} \mathcal{E}_{-\nu}(A) e^{-u \mathcal{F}_2} &= e^{u\ad \mc{F}_2} \mc{E}_{-\nu}(A) = \sum_{k=0}^\infty \frac{u^k}{k!} \ad_{\mc{F}_2}^k \mc{E}_{-\nu}(A)\\
&= \sum_{k=0}^\infty \frac{u^k}{k!} \sum_{l\in \Z'} \Big( \frac{(l+\nu )^2-l^2}{2}\Big)^k e^{A(l+\frac{\nu}{2})} E_{l+\nu, l}\\
&= \sum_{l\in \Z'} \sum_{k=0}^\infty \frac{u^k (l\nu + \frac{\nu^2}{2})^k}{k!} e^{A(l+\frac{\nu}{2})} E_{l+\nu, l}\\
&= \sum_{l\in \Z'} e^{u\nu (l + \frac{\nu}{2})} e^{A(l+\frac{\nu}{2})} E_{l+\nu, l}= \sum_{l\in \Z'} e^{(A+u\nu)(l+\frac{\nu}{2})} E_{l+\nu, l}\,,
 \end{align}
 which coincides with the right-hand side by definition.
\end{proof}
\begin{remark}
For $A=0$, this lemma recovers \cite[Equation 2.14]{OP}.
\end{remark}

\begin{lemma}\label{lem:ConjDE}
\begin{align}
\mathcal{D}^{(h)}(u) \mathcal{E}_{-\nu}(A) \mathcal{D}^{(h)}(u)^{-1} & =\sum_{v=0}^{\infty}\frac{(v+\nu-1)!}{(\nu-1)!}[z^v]\mathcal{S}(uz)^{\nu-1}\mathcal{E}_{-\nu}(A + uz)\,;\\
\mathcal{D}^{(\sigma)}(u) \mathcal{E}_{-\nu}(A) \mathcal{D}^{(\sigma)}(u)^{-1} & =\sum_{v=0}^{\nu}\frac{\nu!}{(\nu-v)!}[z^v]\mathcal{S}(uz)^{-\nu-1}\mathcal{E}_{-\nu}(A + uz)\,.
\end{align}
\end{lemma}
\begin{proof}
This is again a straightforward computation. It can obtained by modifying slightly the proofs of \cite[Lemmata 4.1 and 4.3]{kramer2016quasi}.
\end{proof}

By these two lemmata, we can express the generating function in \cref{TriplyMixedGeneratingConj} as a linear combination of correlators purely in terms of the $\mathcal{E}$-operators and we obtain the following proposition generalising \cref{eq:h_as_wedge,eq:sigma_as_wedge}.

\begin{proposition}
\label{eq:mixedhurwitzgen}
Let $g$ be a non-negative number, $\mu$ and $\nu$ partitions of the same positive integer, and $p$, $q$, $r$ non-negative integers, such that $p+q+r=b$. The triply mixed Hurwitz number corresponding to these data can be computed as
 \begin{align}
 h_{p,q,r;\mu, \nu}^{(2),\leq,<} &= p! [X^p Y^q Z^r] \sum_{v,w \in \N^n} \prod_{j=1}^n \frac{(\nu_j+v_j -1)!}{\mu_j (\nu_j-w_j)!} \cdot \\ 
&\qquad [\underline{y}^v \underline{z}^w] \prod_{j=1}^n \frac{\mc{S}(Yy_j)^{\nu_j-1}}{\mc{S}(Zz_j)^{\nu_j+1}} \bigg\< \prod_{i=1}^m \mc{E}_{\mu_i}(0) \prod_{j=1}^n \mc{E}_{-\nu_j} (X \nu_j +Yy_j+Zz_j) \bigg \> \,.
 \end{align}
 \end{proposition}
 
 Let us analyse this expression. First, the variable \( u\) has been omitted and is replaced by three variables, \( X\), \( Y\), and \( Z\), that count one kind of ramification each. Furthermore, \( Y\) always occurs together with a \( y_j \), and similarly for \( Z\). Hence, the parameter \( q \) on the left-hand side corresponds to \( \sum_{i=1}^n v_n\) on the right-hand side and similarly \( r \) corresponds to \( \sum_{j=1}^n w_n \).

\section{Johnson's algorithm for (strictly) monotone Hurwitz numbers}\label{sec:Johnson}
In this section we apply an algorithm described in \cite{johnson2015} to evaluate the vacuum expectations expressing monotone and strictly monotone Hurwitz numbers.
For $I,K\subset [m]$ and $J,L \subset [n]$, where $[n]= \{1, \dots, n\}$, define
\begin{equation}
\bvs{I}{J}{K}{L}=\varsigma\left(\text{det}\begin{bmatrix} |\mu_I|-|\nu_J| & z_J \\ |\mu_K|-|\nu_L| & z_L \end{bmatrix}\right)\,.\label{eq:linearvarsigma}
\end{equation}
where $\mu_I = \sum_{i \in I} \mu_i$ for a partition $\mu$, and similarly $z_I = \sum_{i \in I} z_i$ for the variables $z_i$.
Define moreover
\begin{equation}
\mathcal{E}'(I, J) = \mathcal{E}_{|\mu_I| - |\nu_J|}(z_J)
\end{equation}
and observe that
\begin{equation}
[\mathcal{E}'(I, J), \mathcal{E}'(K, L)] = \bvs{I}{J}{K}{L} \mathcal{E}'(I \cup K, J \cup L)\,.\label{Ecomm}
\end{equation}

Following \cite[Section 3]{johnson2015}, we choose a chamber \( \mf{c} \) of the hyperplane arrangement \( \mc{W}(m,n)\), and consider the expression
\begin{equation}
\left\langle\prod_{i=1}^{m}\mathcal{E}_{\mu_i}(0)\prod_{j=1}^n\mathcal{E}_{-\nu_j}(z_j)\right\rangle
\end{equation}
there. The idea of the algorithm is to commute all positive-energy operators to the right and all negative-energy operators to the left, where they will annihilate the vacuum and the covacuum, respectively. In doing so, we pick up correlators, reducing the total amount of operators in the correlator. This ensures the algorithm terminates.\par
More explicitly, suppose we have a term of the form
\begin{equation}
\bigg\< \prod_{i=1}^k \mc{E}'(I_i,J_i)\bigg\>\,,\label{corrshape}
\end{equation}
where the product is ordered. Take the left-most negative-energy operator, \( \mc{E}'(I_i,J_i)\). If it is next to the covacuum, the term is zero. Otherwise, commute it to the left. By \cref{Ecomm}, this commutation results in two new terms: one where the factors \( \mc{E}(I_{i-1},J_{i-1}) \) and \( \mc{E}'(I_i,J_i )\) are switched, and one where they are replaced by \( \mc{E}' (I_{i-1} \cup I_i,J_{i-1}\cup J_i )\). Both of these terms are again of shape \cref{corrshape}, so the algorithm can continue.\par
In the end we get the following formula:
\begin{equation}
\left\langle\prod_{i=1}^{m}\mathcal{E}_{\mu_i}(0)\prod_{j=1}^n\mathcal{E}_{-\nu_j}(z_j)\right\rangle = \frac{1}{\varsigma(z_{[n]})} \sum_{P \in CP^{\mathfrak{c}}}^{\text{finite}} \prod_{\ell=1}^{m+n-1} \bvs{I^P_\ell}{J^P_\ell}{K^P_\ell}{L^P_\ell},\label{eq:Joh_decomp}
\end{equation}
where $CP^{\mf{c}}$ is a finite set of \emph{commutation patterns} that only depends on the chamber $\mathfrak{c}$ of the hyperplane arrangement $\mc{W}(m,n)$: the chamber determines the sign of the energy of the \( \mc{E}' \)-operators obtained from the commutators, and hence the operators to be commuted in future steps. The \( I^P_\ell\), \( J^P_\ell\), \( K^P_\ell\), and \( L^P_\ell \) are the four partitions involved in the \( \ell \)-th step of commutation pattern \( P \).\par
Note that the only difference between the correlators on the left-hand side of  \cref{eq:Joh_decomp} and the ones used in Johnson's paper is in the arguments of the $\mathcal{E}'$-operators with negative energy. This difference only affects slightly the definition of the functions $\varsigma '$ and the prefactor $1/\varsigma(z_{[n]})$.\par
Combining \cref{eq:Joh_decomp} with \ref{eq:h_as_wedge} and \ref{eq:sigma_as_wedge} and substituting $u z_j\mapsto z_j$, we have just proved the first main theorem of this paper from which we will derive \cref{thm:piecewise} and \cref{theorem:wallcrossingmon}.

\begin{theorem}
\label{theorem:Joh_alg}
Let $g$ be a non-negative integer and let $m$, $n$ be positive integers such that $(g, n+m) \neq (0,2)$. Let $\mf{c}$ be a chamber of the hyperplane arrangement $\mc{W}(m,n)$. For each $\mu,\nu\in\mf{c}$, we have
\begin{equation}
h_{g;\mu,\nu}^{\le} \! = \frac{1}{\prod \mu_i} \!\! \sum_{\substack{v \vdash b \\ \ell(v)=n}} \prod_{j=1}^n\frac{(v_j\! +\! \nu_j \!- \! 1)!}{\nu_j!}[z^{v_1}_1\dots z^{v_n}_n]\prod_{j=1}^n\mathcal{S}(z_j)^{\nu_j-1}\frac{1}{\varsigma(z_{[n]})} \sum_{P \in CP^{\mathfrak{c}}}^{\textup{finite}} \prod_{\ell=1}^{m+n-1} \bvs{I^P_\ell}{J^P_\ell}{K^P_\ell}{L^P_\ell}\,
\end{equation}
and
\begin{equation}
 h_{g;\mu,\nu}^{<}=\frac{1}{\prod\mu_i}\!\! \sum_{\substack{v \vdash b\\ 0 \leq v_j \leq \nu_j }}\prod_{i=1}^n\frac{(\nu_j-1)!}{(\nu_j-v_j)!}[z^{v_1}_1\dots z^{v_n}_n]\prod_{j=1}^n\mathcal{S}(z_j)^{-\nu_j-1}\frac{1}{\varsigma(z_{[n]})} \sum_{P \in CP^{\mathfrak{c}}}^{\textup{finite}} \prod_{\ell=1}^{m+n-1} \bvs{I^P_\ell}{J^P_\ell}{K^P_\ell}{L^P_\ell}\,.
\end{equation}
\end{theorem}

%%%%%%%%%%%%%%%%%%%%%%%%%%%%%%%%
%%%                         SECTION
%%%%%%%%%%%%%%%%%%%%%%%%%%%%%%%%

\section{Piecewise polynomiality for double Hurwitz numbers}\label{sec:piecewise}

In this section we begin approaching the problem of piecewise polynomiality of triply mixed Hurwitz numbers. We use a  semi-infinite wedge approach to this problem inspired by Johnson's work in \cite{johnson2015}. To be more precise, we begin by deriving piecewise polynomiality for monotone Hurwitz numbers (recovering \cref{theorem:mixhur} for $p=0$) and for strictly monotone Hurwitz numbers directly from the expression in \cref{theorem:Joh_alg}. This shows that triply mixed Hurwitz numbers are piecewise polynomial for the extremal cases of $p=b$, $q=b$, and $r=b$.

\subsection{Piecewise polynomiality for monotone and strictly monotone Hurwitz numbers}\label{piecewisemonotone}

\begin{theorem}[Piecewise polynomiality]\label{thm:piecewise}
Let $g$ be a non-negative integer and let $m$, $n$ be positive integers such that $(g, n+m) \neq (0,2)$. Let $\mf{c}$ be a chamber of the hyperplane arrangement $\mc{W}(m,n)$.
Then there exist polynomials $P_{g}^{\mf{c}, \leq}$ and $P_{g}^{\mf{c}, <}$ of degree \( 4g-3+m+n\) in $m+n$ variables such that 
\begin{align}
h^{\le}_{g;\mu,\nu}&=P_{g}^{\mf{c}, \leq}(\mu,\nu)\,;\\
h^{<}_{g;\mu,\nu}&=P_{g}^{\mf{c}, <}(\mu,\nu)
\end{align}
 for all $(\mu,\nu)\in \mf{c}$.
\end{theorem}

\begin{remark}
The case $(g, n+m) = (0,2)$ only occurs for $g=0$ and $\mu = \nu = (d)$ for some positive integer $d$, which implies that there are no intermediate ramifications ($b=0$). In this case there is, up to isomorphism, a unique covering $z \mapsto \alpha z^d$, for $\alpha \in \C^{\times}$, with automorphism group of order $d$. Hence the Hurwitz number equals $h_{0; (d), (d)} = \frac{1}{d}$ independently of the monotonicity conditions, reflecting a rational function this time, but indeed again of degree $4g - 3 +m + n = -1$.
\end{remark}
\begin{proof} Let us first prove the statement for the monotone case. We fix a chamber $\mf{c}$. By \cref{theorem:Joh_alg}, we can write the monotone Hurwitz numbers as
 \begin{equation}
  \prod_{i=1}^m \mu_i\prod_{j=1}^n \nu_j h_{g;\mu,\nu}^{\le} \! = \!\!\!\! \sum_{\substack{v \vdash b \\ \ell(v)=n}} \! \prod_{j=1}^n\frac{(v_j\! +\! \nu_j \!- \! 1)!}{(\nu_j\! -\! 1)!}[z^{v_1}_1\dots z^{v_n}_n]\prod_{j=1}^n\mathcal{S}(z_j)^{\nu_j-1}\frac{1}{\varsigma(z_{[n]})} \sum_{P \in CP^{\mathfrak{c}}}^{\text{finite}} \prod_{\ell=1}^{m+n-1} \bvs{I^P_\ell}{J^P_\ell}{K^P_\ell}{L^P_\ell}\,. \label{eq:tuamadre}
\end{equation}	
Let us first prove the following:
\begin{lemma}\label{polymonotonefactor}
For $(\mu, \nu) \in \mathfrak{c}$, each summand
\begin{align}
[z^{v_1}_1\dots z^{v_n}_n]\prod_{j=1}^n\mathcal{S}(z_j)^{\nu_j-1}\frac{1}{\varsigma(z_{[n]})} \sum_{P \in CP^{\mathfrak{c}}}^{\textup{finite}} \prod_{\ell=1}^{m+n-1} \bvs{I^P_\ell}{J^P_\ell}{K^P_\ell}{L^P_\ell}
\end{align}
is a polynomial in the entries of $\mu$ and $\nu$ of degree bounded by $2g - 1 + m + n$.
\end{lemma}
\begin{proof}
Let us recall that the expansion of the function $\mathcal{S}(z)$ reads
\begin{equation}
\mathcal{S}(z)=\frac{2\sinh(z/2)}{z} = \sum_{n=0}\frac{z^{2n}}{2^{2n}(2n+1)!} = 1+\frac{z^2}{24}+\frac{z^4}{1920}+\mc{O}(z^6).
\end{equation}
Hence, the coefficient of $z_j^{2t}$ in $\mathcal{S}(z_j)^{\nu_j-1}$ is a polynomial in $\nu_j$ of degree $t$. We show that 
\begin{align}
\frac{1}{\varsigma(z_{[n]})} \sum_{P \in CP^{\mathfrak{c}}}^{\text{finite}} \prod_{\ell=1}^{m+n-1} \bvs{I^P_\ell}{J^P_\ell}{K^P_\ell}{L^P_\ell}
\end{align}
is a formal power series in $z_1, \dots, z_n$: Let $B_k$ be the $k$-th Bernoulli number. The expansion of $1/\varsigma(z)$ reads
\[ \frac{1}{\varsigma(z)} = \frac 1 z - \sum_{n=1}^{\infty} \frac{(1 - 2^{1-2n})B_{2n} z^{2n-1}}{(2n)!} =\frac 1 z - \frac{z}{24} + \frac{7z^3}{5760} + \mc{O}(z^5).\]
Therefore we need to show that $z_{[n]}$ divides the product of the functions $\varsigma'$ in \cref{eq:Joh_decomp} for each commutation pattern $P$. Indeed it suffices to observe that, for every commutation pattern $P$, in the last step of Johnson's algorithm the correlator is 
\begin{equation}
\dcor{\mathcal{E}_a(z_I)\mathcal{E}_{-a}(z_{[n] \setminus I})} = \varsigma(a z_{[n]}) \dcor{\mathcal{E}_0(z_{[n]})}
\end{equation}
for some $I$ and $a$ depending on $P$, which is divisible by $z_{[n]}$. Note that the functions $\varsigma '$, by \cref{eq:linearvarsigma}, are odd functions of either $z_i\mu_j$ or $z_i\nu_j$, for some $i$ and $j$. Therefore the coefficient of $[z^{w_1}_1\dots z^{w_n}_n]$ is a polynomial in $\mu_i$ and $\nu_j$ of degree $w_{[n]} + 1 $. This concludes the proof of the lemma.
\end{proof}

 Now observe that $\frac{(v_j+\nu_j-1)!}{(\nu_j-1)!}$ is a polynomial in $\nu_j$ of degree $v_j$ and lower degree equal to one if $v_j$ is non-zero, hence each $\prod_{j=1}^n\frac{(v_j-\nu_j-1)!}{(\nu_j-1)!}$ is a polynomial in the entries of $\nu$ of degree $2g - 2 + m + n$, and lower degree equal to the number of $v_j$ that are non-zero. 
%\begin{align}
%\label{equ:interm}%some intermediate term
%[u^b]\sum_{v_1=0}^{\infty}\cdots\sum_{v_n=0}^{\infty}\prod_{i=1}^n\frac{(v_i-\nu_i-1)!}{(\nu_i-1)!}[z^{v_1}\dots z^{v_n}]\prod_{i=1}^n\mathcal{S}(uz_i)^{\nu_i-1}\left\langle\prod_{i=1}^m\mathcal{E}_{\mu_i}(0)\prod_{j=1}^n\mathcal{E}_{-\nu_j}(uz_j)\right\rangle
%\end{align}
%is a sum over piecewise polynomial functions. Piecewise polynomiality for the term in \Cref{equ:interm} follows if the number of summands is finite and does not depend on $z_i,\mu_i,\nu_i$ and $u$. This is true, since the total degree of the $z_i$s in \begin{align}
%\prod_{i=1}^n\mathcal{S}(uz_i)^{\nu_i-1}\left\langle\prod_{i=1}^m\mathcal{E}_{\mu_i}(0)\prod_{j=1}^n\mathcal{E}_{-\nu_j}(uz_j)\right\rangle
%\end{align}
%is equal to the degree of $u$. Thus the number of monomials $z^{v_1}\cdots z^{v_n}u^b$ is bounded by the number of tuples $\vec{v}=(v_1,\dots,v_n)$, such that $|\vec{v}|=b=2g-2+m+n$ which does not depend on any of the variables $z_i,\mu_i,\nu_i,u$. Moreover, the coefficient of each monomial is a polynomial in the $\mu_i,\nu_j$. 
This implies the piecewise polynomiality for \(\prod\mu_i\prod\nu_j h^{\le}_{g;\mu,\nu} \). We are left to prove that each $\mu_i$ and each $\nu_j$ divides the right-hand side of \cref{eq:tuamadre}. For the divisibility by $\mu_i$, observe that, since $\mathcal{E}_{\mu_i}(0)$ has positive energy, in any commutation pattern $P$ it happens that it is commuted with an operator of the form $\mathcal{E}_{\mu_K-\nu_L}( z_L)$ producing a factor $\varsigma(\mu_i z_L)$, which is divisible by $\mu_i$. To prove the divisibility by $\nu_j$, we distinguish two cases:
\begin{description}
\item[$v_j\neq0$] In this case the factor $\frac{(v_j+\nu_j-1)!}{(\nu_j-1)!} $ is divisible by $\nu_j$;
\item[$v_j=0$] Since the operator $\mathcal{E}_{-\nu_j}(z_j)$ has negative energy, in any commutation pattern $P$ it happens that it is commuted with an operator of the form $ \mathcal{E}_{\mu_K-\nu_L}(\sum z_L)$ producing a factor 
$$
\varsigma\left((\mu_K - \nu_L)z_j - \nu_j z_L\right),
$$
hence the coefficient of $[z_j^0]$ of the corresponding summand is divisible by $\nu_j$.
\end{description}
Note that the division by the factor $\prod \mu_i \prod \nu_j$ decreases the degree of the polynomial by $n+m$. Hence, the total upper bound for the degree of the polynomial $P^{\mf{c}, \leq}_g$ is \[(2g - 1 + m + n) + (2g - 2 + m + n) - (m + n) = 4g  - 3 + m + n, \]
while the lower bound is given by $(m+n-1) + 1 - (m + n) = 0$. This concludes the proof for the monotone case.

Let us now prove the strictly monotone case. By \cref{eq:sigma_as_wedge} we can rewrite the strictly monotone Hurwitz numbers as
 \begin{equation}
 \prod_{i=1}^m \mu_i\prod_{j=1}^n \nu_j h_{g;\mu,\nu}^{<}=\!\!\!\! \sum_{\substack{v \vdash b\\ 0 \leq v_j \leq \nu_j }} \!\! \prod_{i=1}^n\frac{(\nu_j-1)!}{(\nu_j-v_j)!}[z^{v_1}_1\dots z^{v_n}_n]\prod_{j=1}^n\mathcal{S}(z_j)^{-\nu_j-1}\frac{1}{\varsigma(z_{[n]})} \sum_{P \in CP^{\mathfrak{c}}}^{\text{finite}} \! \prod_{\ell=1}^{m+n-1} \! \bvs{I^P_\ell}{J^P_\ell}{K^P_\ell}{L^P_\ell}\,.\label{eq:tuasorella}
\end{equation}
Note that the only differences with the monotone case are in the powers of the functions $\mathcal{S}$ and in the prefactor $\frac{\nu!}{(\nu - v)!}$. However, the coefficient of $z^{2t}$ in $\mathcal{S}^{-\nu-1}$ is again a polynomial in $\nu$ of degree $t$, and the prefactor $\frac{\nu!}{(\nu - v)!}$ is again a polynomial in $\nu$ of degree $v_i$. Therefore the entire same argument applies with the same lower and upper bounds on the degrees. This concludes the proof of \cref{thm:piecewise}.
\end{proof}

\subsection{An example: computing the lowest degree for the monotone case}\label{lowestdegree}
Let us test our formula computing the lowest degree for the monotone case. Firstly, note that, because the factor $\frac{(v_j+\nu_j-1)!}{(\nu_j-1)!} $ is divisible by $\nu_j$ for $v_j \neq 0$, the lowest degree occurs for all $v_j = 0$ but one. Hence let us consider vectors $v = (0, \dots, b, \dots, 0)$, for $b$ in the $k$-th position for some $k = 1, \dots, n$. Then the expression for the monotone case then reads
 \begin{equation}  
[\deg_{\nu,\mu}\!=0] \frac{(b+\nu_k-1)!}{\prod\mu_i\prod\nu_j(\nu_k-1)!}[z^{0}_1\dots z^{b}_k \dots z^0_n]\prod_{i=1}^n\mathcal{S}(z_i)^{\nu_i-1}\left\langle\prod_{i=1}^m\mathcal{E}_{\mu_i}(0)\prod_{j=1}^n\mathcal{E}_{-\nu_j}(z_j)\right\rangle.
\end{equation}
We can therefore set $z_j=0$ for $j \neq k$. This implies that there is only one possible commutation pattern. Summing over $k$ we obtain that the total lowest degree is
\begin{equation}
[\deg_{\nu,\mu} \!=0]\sum_{k=1}^{n}(b+\nu_k - 1)\dots (\nu_k +1)  [z^{2g-2 + m +n}] \prod_{\substack{j =1 \\ j \neq k}}^n \frac{\varsigma(z \nu_j)}{\nu_j}\prod_{i = 1}^m \frac{\varsigma(z \mu_i)}{\mu_i} \frac{\mathcal{S}(z)^{\nu_k - 1}}{\varsigma(z)}
\end{equation}
In order to compute the lowest degree, we have to pick the linear term from each $\varsigma$-function at the numerator, hence we find that
\begin{equation}
 [\deg_{\nu,\mu} \!=0]h_{g;\mu,\nu}^{\le} = (b-1)! \sum_{k=1}^{n}  [\deg_{\nu} \!=0][z^{2g-2}] \mathcal{S}(z)^{\nu - 2}\,.
\end{equation}
%If $\nu = 1$, then 
%\begin{equation}
%[z^{2g-2}]\mathcal{S}(z)^{-1} = \frac{(2^{3-2g}-1)B_{2g-2}}{(2g-2)!}\delta_{g \geq2} + \delta_{g,1},
%\end{equation}
%whereas if $\nu = 2$, we have $[z^{2g-2}]\mathcal{S}(z)^{0} = \delta_{g,1}$. Let us consider now the case $\nu \geq 3$ and $g \geq 1$. In this case we want to select the non-negative even power of $z$ equal to $2g-2$ from a positive power of the even function $\mathcal{S}(z)$. 
Recall the generating series of the generalised Bernoulli polynomials \( B_k^{(n)}(x)\)\cite[p. 145]{Norlund1924} (cf. also \cite[Section 4.2.2]{Roman1984}), by
\begin{equation}
\bigg(\frac{t}{e^t-1}\bigg)^n e^{xt} \eqcolon \sum_{k=0}^\infty B_k^{(n)}(x)\frac{t^k}{k!}\,,
\end{equation}
with specific cases given by \( B_k^{(n)} \coloneq B_k^{(n)}(0)\) and the standard Bernoulli numbers \( B_k \coloneq B_k^{(1)} \) (with \( B_1 = -\frac{1}{2} \)). These are polynomial in both \( n\) and \( x\). In our case, this gives
\begin{align}
[z^{2g-2}].\mc{S}(z)^{\nu-2} &= [z^{2g-2}].\bigg( \frac{e^z-1}{z}\bigg)^{\nu-2}e^{-\frac{\nu-2}{2}z} = \frac{B_{2g-2}^{(2-\nu)}\big(\frac{2-\nu}{2}\big)}{(2g-2)!}\\
&= \frac{1}{(2g-2)!}\sum_{k=0}^{2g-2} \binom{2g-2}{k} \Big( \frac{2-\nu}{2}\Big)^{2g-2-k} B_k^{(2-\nu )}\,.
\end{align}
Taking the degree zero part in \( \nu \) corresponds to setting \( \nu =0 \), which yields, using \cite[Equation 81*]{Norlund1924},
\begin{align}
\sum_{k=0}^{2g-2} \frac{1}{k!(2g-2-k)!} B_{k}^{(2)} &= \sum_{k=0}^{2g-2} \frac{1}{k!(2g-2-k)!}\big((1-k)B_k-kB_{k-1}\big)\\
%&= -\bigg( \sum_{k=0}^{2g-2} \frac{1}{k!(2g-2-k)!}\big((k-1)B_k+kB_{k-1}\big)\bigg)\\
&= -\bigg( \sum_{k=0}^{2g-2} \frac{(k-1)B_k}{k!(2g-2-k)!}+\frac{B_{k-1}}{(k-1)!(2g-2-k)!} \bigg) \\
&= -\sum_{k=0}^{2g-2} \frac{(k-1)B_k}{k!(2g-2-k)!}-\sum_{k=0}^{2g-3}\frac{(2g-2-k)B_k}{k!(2g-2-k)!} \\
&=  -\frac{2g-3}{(2g-2)!}\sum_{k=0}^{2g-2} \binom{2g-2}{k}B_k= -\frac{(2g-3)B_{2g-2}}{(2g-2)!}
\end{align}

Hence, the final expression reads
\begin{equation}
[\deg_{\nu,\mu} = 0]h_{g;\mu,\nu}^{\leq} = -\frac{n \left(2g - 3 + m + n \right)!(2g-3)B_{2g-2}}{(2g-2)!} \delta_{g\geq 1},
\end{equation}
which shows that the lowest degree does not vanish for $g \geq 1$.

\subsection{Piecewise polynomiality for triply mixed Hurwitz numbers}\label{piecewisemixed}

After having developed the necessary tools in \cref{piecewisemonotone}, we use the same approach to prove piecewise polynomiality of triply mixed Hurwitz numbers in this section. We use the expression for triply mixed Hurwitz numbers in \cref{eq:mixedhurwitzgen}.

%In this section we prove piecewise polynomiality of double Hurwitz numbers with several possible combinations of ramifications. More precisely, we define mixed Hurwitz numbers with additional parameters whose exponents encode the number of ramifications of a certain kind (simple, monotone, strictly monotone). Setting the relevant parameter to zero then recovers the Hurwitz numbers without that kind of ramification.\par
 
 \begin{theorem}[Piecewise polynomiality for triply mixed Hurwitz]\label{thm:mixedpiecewise}
Let $p,q,r$ be non-negative integers and let $m$, $n$ be positive integers such that $(g, n+m) \neq (0,2)$, where \( p+q+r = 2g-2 + m+ n\). Let $\mf{c}$ be a chamber of the hyperplane arrangement $\mc{W}(m,n)$.
Then there exist polynomials $P_{p,q,r}^{\mf{c}; (2),\leq,<}$ of degree \( 4g-3+m+n \) in $m+n$ variables such that 
\begin{equation}
h^{(2),\leq,<}_{p,q,r;\mu,\nu}=P_{p,q,r}^{\mf{c}; (2),\leq,<}(\mu,\nu)
\end{equation}
 for all $(\mu,\nu)\in \mf{c}$.
\end{theorem}
\begin{remark}
Notice that \cref{thm:piecewise} is a special case of this theorem, obtained by setting \( p\) and either \( r\) or \( q\) to zero. Likewise, the mixed cases of two out of the three kinds of Hurwitz number can be obtained by setting the third parameter to zero. In particular, we recover \cref{theorem:mixhur} by setting $r=0$.
\end{remark}

\begin{proof}
In \cref{eq:mixedhurwitzgen}, let us first look at a single factor
\begin{equation}
[X^p\vec{y}^v \vec{z}^w] \prod_{j=1}^n \frac{\mc{S}(y_j)^{\nu_j-1}}{\mc{S}(z_j)^{\nu_j+1}} \bigg\< \prod_{i=1}^m \mc{E}_{\mu_i}(0) \prod_{j=1}^n \mc{E}_{-\nu_j} (X \nu_j +y_j+z_j) \bigg \>\,,
\end{equation}
where \( |v| = q\) and \( |w| = r\).\par
Because \( \mc{S} (z)\) is an even analytic function with constant term \( 1\) and non-zero coefficient of \( z^2 \), the coefficient of \( z^{2t} \) in both \( \mc{S}(z)^{\nu-1} \) and \( \mc{S}(z)^{-\nu-1} \) is a polynomial in \( \nu \) of degree \( t\). On the other hand, the commutations produce factors where every factor of \( y \) or \( z \) brings a linear polynomial in \( \nu \) and \( \mu \) and every factor of \( X\) brings a quadratic polynomial. As the final correlator of the commutation pattern still gives a factor \( \varsigma (X\nu_{[n]} + y_{[n]} + z_{[n]})^{-1} \), this complete factor gives a polynomial in \( \mu \) and \( \nu \) of degree \( 2p + q+r+1\).\par
The correlator can be calculated using Johnson's algorithm, where the set of commutation patterns is fixed by the chamber \( \mathfrak{c} \). Every commutation gives a factor of \( \varsigma \) with a certain argument linear or quadratic in \( \mu \) and \( \nu \), until we end up with 
\begin{equation}
\Big\< \mc{E}_a(X\nu_{I} + y_{I} + z_{I}) \mc{E}_{-a}(X\nu_{[n]\setminus I} + y_{[n]\setminus I} + z_{[n]\setminus I})\Big\> 
\end{equation}
for some \( a \geq 0\) and \( I \subset [n]\). By the commutation rules, this is equal to
\begin{equation}
\varsigma \big( a(X\nu_{[n]} + y_{[n]} + z_{[n]}) \big) \Big\<  \mc{E}_{0}(X\nu_{[n]} + y_{[n]} + z_{[n]})\Big\>  =\frac{\varsigma \big( a(X\nu_{[n]} + y_{[n]} + z_{[n]}) \big)}{\varsigma (X\nu_{[n]} + y_{[n]} + z_{[n]} )}\,.
\end{equation}
The possible pole coming from the denominator is cancelled by the numerator, so this entire term is polynomial in \( \mu \) and \( \nu \).\par
Furthermore, this polynomial is divisible by \( \mu_i \), as the operator \( \mc{E}_{\mu_i}(0) \) must be commuted with some negative-energy operator \( \mc{E}_{-a}(x)\), producing a factor \( \varsigma (\mu_i x)\).\par
Also, the factor \( \frac{(\nu_j + v_j-1)!}{(\nu_j -w_j)!} \) is polynomial in \( \nu \) -- of degree \( v_j + w_j -1\) -- unless \( v_j = w_j = 0\), in which case it is \( \frac{1}{\nu_j}\). However, in this case we have the operator \( \mc{E}_{-\nu_j}(X\nu_j) \), which must commute to the left, and will always yield some factor \( \varsigma (\nu_j x) \) in the commutator. Hence, the entire term
\begin{equation}
[X^p]\prod_{j=1}^n \frac{(\nu_j+v_j -1)!}{\mu_j (\nu_j-w_j)!} [\vec{y}^v \vec{z}^w] \prod_{j=1}^n \frac{\mc{S}(y_j)^{\nu_j-1}}{\mc{S}(z_j)^{\nu_j+1}} \bigg\< \prod_{i=1}^m \mc{E}_{\mu_i}(0) \prod_{j=1}^n \mc{E}_{-\nu_j} (X \nu_j +y_j+z_j) \bigg \>
\end{equation}
is polynomial in \( \mu \) and \( \nu \).\par
To calculate the coefficient of \( X^pY^qZ^r \) in \cref{eq:mixedhurwitzgen}, we take a \emph{finite} sum over such polynomials, where the number of summand is independent of \( \mu \) and \( \nu \), as the sum runs over non-negative \( \{ v_i,w_i \mid 1 \leq i \leq n\} \) such that \( \sum_i v_i = q \) and \( \sum_i w_i = r\).\par
The maximal degree of this polynomial is then
\begin{equation}
(2p+q+r+1) + \sum_{j=1}^n (v_j+w_j -1 ) - m = 2(p+q+r)+1-n-m = 4g-3+m+n\,,
\end{equation}
which proves the theorem.
\end{proof}

The lower bound of the polynomial corresponds to the power of $X$ that we choose in the polynomial $h_{g;\mu,\nu}^{<, \le, (2)}(X,Y,Z)$, since the powers of $X$ do not come from any other expansion.

%%%%%%%%%%%%%%%%%%%%%%%%%%%%%%%%
%%%                               SECTION
%%%%%%%%%%%%%%%%%%%%%%%%%%%%%%%%

\section{Wall-crossing formulae}\label{sec:wall}

In the previous sections, we have given an explicit way of computing polynomials representing strictly and weakly monotone and simple Hurwitz numbers, or any mix of the three, within a chamber of the hyperplane arrangement. In this section, we show how these different polynomials are connected via wall-crossing formulas, expressing the difference between generating functions in adjacent chambers recursively as a product of two generating functions of Hurwitz numbers of similar kind.

\subsection{Wall-crossing formulae for dessins d'enfant and monotone Hurwitz numbers}
In this section, we study the wall-crossing behaviour of the Hurwitz numbers $h^{\le}_{g;\mu,\nu}$ and \( h^<_{g;\mu,\nu}\). We write \( h_{g;\mu, \nu}^{\bullet} \) in the following to mean either of them, and similarly for related quantities. Let $\mf{c}_1$ and $\mf{c}_2$ be two chambers in the hyperplane arrangement given by $\mathcal{W}$ that are separated by the wall $\delta\coloneq \mu_I-\nu_J=0$. Without loss of generality, we assume that $\delta>0$ on $\mf{c}_2$ and $\delta<0$ on $\mf{c}_1$. Let $p^{\mf{c}_i}_{g;\mu,\nu}$ be the polynomial expressing $h_{g;\mu,\nu}^{\bullet}$ in $\mf{c}_i$. The goal of this section is to compute the wall-crossing at $\delta=0 $ between \( \mf{c}_1 \) and \( \mf{c}_2 \)
\begin{equation}
WC_{\delta}:=p^{\mf{c}_2}_{g;\mu,\nu}-p^{\mf{c}_1}_{g;\mu,\nu}\in\mathbb{Q}[\mu,\nu]\,.
\end{equation}
Our approach to the wall-crossing is motivated by the expression of $h_{g;\mu,\nu}^{\bullet}$ in \cref{theorem:Joh_alg}. 

\begin{notation}
For a partition $\mu$, a subset $I\subset\{1,\dots,m\}$, and a wall $\delta =0$, we denote the partition $(\mu_i)_{i\in I}$ by $\mu^I$ and the partition $(\mu,\delta)$ by $\mu+\delta$, whereas the notation $\mu_I$ is still reserved for $\sum_{i \in I} \mu_i$. Moreover, for a collection of variables $\underline{u}=u_1,\dots,u_n$ and a subset $J\subset\{1,\dots,n\}$, we denote the collection $(u_j)_{j\in J}$ by $u^J$.
\end{notation}

\begin{definition}
Let $\mu,\nu$ be ordered partitions of the same natural number. We define the \textit{refined monotone generating series} as
\begin{equation}
\mathcal{H}^{\le}_{\mu,\nu}(\underline{u},\underline{z})=\sum_{v_1,\dots,v_n=0}^{\infty}u_1^{v_1}\cdots u_n^{v_n}\prod_{j=1}^n\frac{(v_j+\nu_j-1)!}{(\nu_j-1)!}\prod_{j=1}^n\mathcal{S}(z_j)^{\nu_j-1}\frac{\left\langle\prod_{i=1}^{m}\mathcal{E}_{\mu_i}(0)\prod_{j=1}^n\mathcal{E}_{-\nu_j}(z_j)\right\rangle}{\prod_{i=1}^m\mu_i\prod_{j=1}^n\nu_j}\,.\label{equ:refined}
\end{equation}
Similarly, we define the \textit{refined Grothendieck dessins d'enfant generating series} as
\begin{equation}
\mathcal{H}^{<}_{\mu,\nu}(\underline{u},\underline{z})=\sum_{\substack{v_1,\dots,v_n=0 \\ 0 \leq v_i \leq \nu_i}}^{\infty}u_1^{v_1}\cdots u_n^{v_n}\prod_{j=1}^n\frac{\nu_j!}{(\nu_j-v_j)!}\prod_{j=1}^n\mathcal{S}(z_j)^{-\nu_j-1}\frac{\left\langle\prod_{i=1}^{m}\mathcal{E}_{\mu_i}(0)\prod_{j=1}^n\mathcal{E}_{-\nu_j}(z_j)\right\rangle}{\prod_{i=1}^m\mu_i\prod_{j=1}^n\nu_j}\,.
\end{equation}
\end{definition}

The following lemma follows from \cref{eq:tuamadre}.
\begin{lemma}\label{hasrefinedgencoef}
Let $g$ be a non-negative integer, $\mu,\nu$ ordered partitions of the same natural number and $b=2g-2+\ell(\mu)+\ell(\nu)$. Then
\begin{equation}
h^{\bullet}_{g;\mu,\nu}=\sum_{\substack{v_1,\dots,v_n\in\mathbb{Z}_{\ge0}\\ |\vec{v}|=b}}[z_1^{v_1}\cdots z_n^{v_n}] [u_1^{v_1}\cdots u_n^{v_n}]\mathcal{H}^{\bullet}_{\mu,\nu}(\underline{u},\underline{z})\,.
\end{equation}
%and
%\begin{align}
%h^{<}_{g;\mu,\nu}=\sum_{\substack{v_1,\dots,v_n\in\mathbb{Z}_{\ge0},\\ |\vec{v}|=b}}[z_1^{v_1}\cdots z_n^{v_n}]\left([u_1^{v_1}\cdots u_n^{v_n}]\mathcal{H}^{<}_{\mu,\nu}(\underline{u},\underline{z})\right).
%\end{align}
\end{lemma}

By \cref{thm:piecewise}, the polynomial expressing
\begin{equation}
\bigg\langle\prod_{i=1}^{m}\mathcal{E}_{\mu_i}(0)\prod_{j=1}^n\mathcal{E}_{-\nu_j}(z_j)\bigg\rangle\label{equ:locterm}
\end{equation}
in \cref{equ:refined} only depends on the chamber $\mf{c}$ given by $\mathcal{W}$, which motivates the following definition.

\begin{definition}
Let $\mf{c}$ be a chamber induced by the hyperplane arrangement $\mathcal{W}$ and denote by $q^{\mf{c}}(\underline{z})$ the polynomial expressing \cref{equ:locterm} in $\mf{c}$. Then we define
\begin{equation}
\mathcal{H}^{\le}_{\mu,\nu}(\mf{c},\underline{u},\underline{z})=\sum_{v_1,\dots,v_n=0}^{\infty}u_1^{v_1}\cdots u_n^{v_n}\prod_{j=1}^n\frac{(v_j+\nu_j-1)!}{(\nu_j-1)!}\prod_{j=1}^n\mathcal{S}(z_j)^{\nu_j-1}\frac{q^{\mf{c}}(\underline{z})}{\prod_{i=1}^m\mu_i\prod_{j=1}^n\nu_j}.\label{equ:refincham}
\end{equation}
and
\begin{equation}
\mathcal{H}^{<}_{\mu,\nu}(\mf{c},\underline{u},\underline{z})=\sum_{\substack{v_1,\dots,v_n=0 \\ 0 \leq v_i \leq \nu_i}}^{\infty}u_1^{v_1}\cdots u_n^{v_n}\prod_{j=1}^n\frac{\nu_j!}{(\nu_j-v_j)!}\prod_{j=1}^n\mathcal{S}(z_j)^{-\nu_j-1}\frac{q^{\mf{c}}(\underline{z})}{\prod_{i=1}^m\mu_i\prod_{j=1}^n\nu_j}.
\end{equation}
Let $\delta=\mu_I-\nu_J$ for some fixed $I\subset\{1,\dots,m\},J\subset\{1,\dots,n\}$. This defines a wall in $\mathcal{W}$ by $\delta=0$. Let $\mf{c}_1$ and $\mf{c}_2$ be chambers which are seperated by $\delta=0$ and contain $\delta=0$ as a codimension one subspace. Then we define the wall-crossings by
\begin{equation}
\mathcal{WC}^\bullet_{\delta}(\underline{u},\underline{z})=\mathcal{H}^\bullet_{\mu,\nu}(\mf{c}_2,\underline{u},\underline{z})-\mathcal{H}^\bullet_{\mu,\nu}(\mf{c}_1,\underline{u},\underline{z})\,.\label{equ:wacro}
\end{equation}
%and
%\begin{equation}
%\mathcal{WC}^{<}_{\delta}(\underline{u},\underline{z})=\mathcal{H}^{<}_{\mu,\nu}(\mf{c}_2,\underline{u},\underline{z})-\mathcal{H}^{<}_{\mu,\nu}(\mf{c}_1,\underline{u},\underline{z})
%\end{equation}
\end{definition}

The following lemma follows from \cref{hasrefinedgencoef}.
\begin{lemma}
Let $g$ be a non-negative integer, $\mu,\nu$ ordered partitions of the same natural number and $b=2g-2+\ell(\mu)+\ell(\nu)$. Then
%\begin{equation}
%WC^{\le}_{\delta}=\sum_{\substack{v_1,\dots,v_n\in\mathbb{Z}_{\ge0}\\ |\vec{v}|=b}}[z_1^{v_1}\cdots z_n^{v_n}] [u_1^{v_1}\cdots u_n^{v_n}]\mathcal{WC}_{\delta}(\underline{u},\underline{z})
%\end{equation}
%and
\begin{equation}
WC^\bullet_{\delta}=\sum_{\substack{v_1,\dots,v_n\in\mathbb{Z}_{\ge0}\\ |\vec{v}|=b}}[z_1^{v_1}\cdots z_n^{v_n}] [u_1^{v_1}\cdots u_n^{v_n}]\mathcal{WC}^\bullet_{\delta}(\underline{u},\underline{z})\,.
\end{equation}
\end{lemma}

The main result of this subsection is the following theorem.

\begin{theorem}
\label{theorem:wallcrossingmon}
Let $\mu,\nu$ be ordered partitions of the same positive integer and let $\delta=\sum_{i\in I}\mu_i-\sum_{j\in J}\nu_j$. Then we have the following recursive structures
\begin{equation}
\mathcal{WC}^\le_\delta (\underline{u},\underline{z})= \delta^2\frac{\varsigma (z_J)\varsigma (z_{J^c})\varsigma (\delta z_{[n]})}{\varsigma (\delta z_J)\varsigma (\delta z_{J^c} )\varsigma (z_{[n]})} [(u')^0]\mathcal{H}^{\le}_{\mu^I,\nu^J+\delta}(\underline{u}^J,u',\underline{z}^J,0)\mathcal{H}^{\le}_{\mu^{I^c}+\delta,\nu^{J^c}}(\underline{u}^{J^c},\underline{z}^{J^c})
\end{equation}
and
\begin{equation}
\mathcal{WC}^{<}_{\delta}(\underline{u},\underline{z})= \delta^2\frac{\varsigma (z_J)\varsigma (z_{J^c})\varsigma (\delta z_{[n]})}{\varsigma (\delta z_J)\varsigma (\delta z_{J^c} )\varsigma (z_{[n]})} [(u')^0]\mathcal{H}^{<}_{\mu^I,\nu^J+\delta}(\underline{u}^J,u',\underline{z}^J,0)\mathcal{H}^{<}_{\mu^{I^c}+\delta,\nu^{J^c}}(\underline{u}^{J^c},\underline{z}^{J^c})\,.
\end{equation}
Here, the argument \( 0\) is the \( z\)-variable related to \( \delta \) in \( \mc{H}_{\mu^I,\nu^J+\delta} \).
\end{theorem}

\begin{proof}
Both formulae are derived by similar calculations, so we only prove the recursive structure for $\mathcal{WC}^{\le}_{\delta}$. The strategy of the proof consists of comparing the generating series $\mathcal{WC}^{\le}_{\delta}$ and $\mathcal{H}^{\le}_{\mu^I,\nu^J+\delta}(\underline{u}^J,\underline{z}^J,z')\mathcal{H}^{\le}_{\mu^{I^c}+\delta,\nu^J}(\underline{u}^{J^c},\underline{z}^{J^c})$, using Johnson's algorithm. We start by studying $\mathcal{WC}^{\le}_{\delta}$. Substituting \cref{equ:refincham} into \cref{equ:wacro}, we obtain
\begin{equation}
\mathcal{WC}^{\le}_{\delta}=\sum_{v_1,\dots,v_n=0}^{\infty}u_1^{v_1}\cdots u_n^{v_n}\prod_{j=1}^n\frac{(v_j+\nu_j-1)!}{(\nu_j-1)!}\prod_{j=1}^n\mathcal{S}(z_j)^{\nu_j-1}\frac{q^{\mf{c}_2}(\underline{z})-q^{\mf{c}_1}(\underline{z})}{\prod_{i=1}^m\mu_i\prod_{j=1}^n\nu_j}\,.\label{equ:diff}
\end{equation}
Let us compute the difference $q^{\mf{c}_2}(z_1,\dots,z_n)-q^{\mf{c}_1}(z_1,\dots,z_n)$. This quantity is almost the same as the one appearing in the proof of the wall-crossing formula for double Hurwitz numbers in \cite[Section 4.2]{johnson2015}. We follow the idea of that proof, making the required adjustments. The main difference is that the vacuum expection 
\begin{equation}
\bigg\langle\prod_{i=1}^m\mathcal{E}_{\mu_i}(0)\prod_{j=1}^n\mathcal{E}_{-\nu_j}(z_j)\bigg\rangle \label{equ:wallterm}
\end{equation}
we consider depends on several variables $z_j$ (one for each entry of $\nu$), whereas the vacuum expectation in \cite{johnson2015}
\begin{equation}
\bigg\langle\prod_{i=1}^m\mathcal{E}_{\mu_i}(0)\prod_{j=1}^n\mathcal{E}_{-\nu_j}(\nu_jz)\bigg\rangle
\end{equation}
only depends on one variable $z$.\par
 Let us first observe that every commutation pattern in which no operator of energy $\delta$ is produced is a summand in both $q^{\mf{c}_2}$ and $q^{\mf{c}_1}$, and therefore contributes trivially to their difference. Thus, it is sufficient to compute the contribution of those commutation patterns producing $\delta$ energy operators. Let us choose the following ordering of operators in the vacuum expectation 
\begin{equation}
\bigg\langle\prod_{i\notin I}\mathcal{E}_{\mu_i}(0)\prod_{i\in I}\mathcal{E}_{\mu_i}(0)\prod_{j\in J}\mathcal{E}_{-\nu_j}(z_j)\prod_{j\notin J}\mathcal{E}_{-\nu_j}(z_J)\bigg\rangle\,.\label{equ:desiredvac}
\end{equation}
If a commutation pattern produces an operator of energy $\delta$, the first vacuum expectation containing that operator must be
\begin{equation}
\bigg\langle\prod_{i\notin I}\mathcal{E}_{\mu_i}(0)\mathcal{E}_{\delta}(z_J)\prod_{j\notin J}\mathcal{E}_{-\nu_j}(z_j)\bigg\rangle\,.\label{equ:wallexp}
\end{equation}
 Let $T_1$ be the product of $\varsigma $-factors the algorithm produces up to \cref{equ:wallexp}. Let us observe that, up until \cref{equ:wallexp}, the algorithm runs identically on $\mf{c}_1$ and $\mf{c}_2$. Therefore, $T_1$ divides $q^{\mf{c}_2} - q^{\mf{c}_1}$. In order to compute the quantity $T_1$, we consider the vacuum expectation
\begin{equation}
\bigg\langle\prod_{i\in I}\mathcal{E}_{\mu_i}(0)\prod_{j\in J}\mathcal{E}_{-\nu_j}(z_j)\mathcal{E}_{-\delta}(0)\bigg\rangle
\end{equation}
inside the chamber $\mf{c}_1$. We claim that the operator $\mathcal{E}_{-\delta}(0)$ cannot be involved in a commutator leading to a non-zero vacuum expectation until the very last commutator. Clearly, the commutator with any negative energy operator is equal to zero. Suppose therefore that $\mathcal{E}_{-\delta}(0)$ is involved in the commutator with some operator 
\begin{equation}
\mathcal{E}_{\mu_K -\nu_L}\big( z_L\big)\,,
\end{equation}
for subsets $K\subset I$ and $L\subset J$, where at least one is a proper subset. Because we are inside a chamber, we have $\mu_K-\nu_L \neq 0$. Hence we assume $\mu_K-\nu_L > 0$. Since we also assumed that the vacuum expectation does not vanish, the commutator must have negative energy. Hence $\mu_K - \nu_L - \delta<0$, which implies 
\begin{equation}
\delta > \mu_K - \nu_L > 0\,.
\end{equation}
This provides a lower bound for $\delta$, contradicting the fact that the chamber $\mf{c}_1$ borders $\delta=0$. We showed that every commutation pattern contributing nontrivially commutes $\mathcal{E}_{-\delta}(0)$ at the very end. Thus all the other commutators must be computed first. Therefore we can compute
\begin{equation}
\bigg\langle \prod_{i\in I}\mathcal{E}_{\mu_i}(0)\prod_{j\in J}\mathcal{E}_{-\nu_j}(z_j)\mathcal{E}_{-\delta}(0) \bigg\rangle =T_1\bigg\langle \mathcal{E}_{\delta}(z_J ) \mathcal{E}_{-\delta}(0) \bigg\rangle =T_1\varsigma(\delta z_J) \bigg\langle\mathcal{E}_0 (z_J) \bigg\rangle =T_1\frac{\varsigma(\delta z_J)}{\varsigma (z_J)}\,.
\end{equation}
Re-arranging the equation, we obtain
\begin{equation}
T_1=\frac{\varsigma (z_J)}{\varsigma(\delta z_J)}\bigg\langle\prod_{i\in I}\mathcal{E}_{\mu_i}(0)\prod_{j\in J}\mathcal{E}_{-\nu_j}(z_j)\mathcal{E}_{-\delta}(0)\bigg\rangle.
\end{equation}
%We want to stress that $T_1$ does not dependend on the auxiliary variable $z'$. 
The quantity in \cref{equ:desiredvac} is therefore
\begin{align}
&\bigg\langle\prod_{i\notin I}\mathcal{E}_{\mu_i}(0)\prod_{i\in I}\mathcal{E}_{\mu_i}(0)\prod_{j\in J}\mathcal{E}_{-\nu_j}(z_j)\prod_{j\notin J}\mathcal{E}_{-\nu_j}(z_j)\bigg\rangle=T_1\bigg\langle\prod_{i\notin I}\mathcal{E}_{\mu_i}(0)\mathcal{E}_{\delta}(z_J)\prod_{j\notin J}\mathcal{E}_{-\nu_j}(z_j)\bigg\rangle\\
&\qquad =\frac{\varsigma (z_J)}{\varsigma(\delta z_J)}\bigg\langle\prod_{i\in I}\mathcal{E}_{\mu_i}(0)\prod_{j\in J}\mathcal{E}_{-\nu_j}(z_j) \mathcal{E}_{-\delta}(0)\bigg\rangle\bigg\langle\prod_{i\notin I}\mathcal{E}_{\mu_i}(0)\mathcal{E}_{\delta}(z_J)\prod_{j\notin J}\mathcal{E}_{-\nu_j}(z_j)\bigg\rangle\,.
\end{align}
We will compare the last factor containing an operator of energy $\delta$ with the vacuum expectation
\begin{equation}
\bigg\langle\mathcal{E}_{\delta}(0)\prod_{i\notin I}\mathcal{E}_{\mu_i}(0)\prod_{j\notin J}\mathcal{E}_{-\nu_j}(z_j)\bigg\rangle\,.\label{equ:comparet2}
\end{equation}
Let $T_2$ be the series denoting the difference of the vacuum expectation \cref{equ:wallexp} on $\mf{c}_2$ and $\mf{c}_1$. Applying Johnson's algorithm to \cref{equ:wallexp}, the operator of energy $\delta$ would be commuted into different directions in the very first step. In order to compare the contributions in each chamber, we commute \( \mathcal{E}_{\delta}(z_J) \) to the left in both chambers, even though it has positive energy on $\mf{c}_2$. If this operator is involved in a cancelling term as we move to the left, the algorithm will run as usual in both chambers after this commutator: after the cancellation, we will have an operator \( \mathcal{E}_{\mu_{K\sqcup I}-\nu_{L\sqcup J}}\big(z_{L\sqcup J}\big)\), where at least one the subsets $K$ and $L$ is non-empty. All contributions up to the cancellation coincide in both chambers (since we chose to commute \( \mathcal{E}_{\delta}(z_J) \) to the left) and by the above argument above so do the contributions after the cancellation. Therefore, we have the same contributions in both chambers with the same sign and they cancel in the wall-crossing.\\
The key observation in computing the difference between $\mf{c}_2$ and $\mf{c}_1$ is that, whenever \( \mathcal{E}_{\delta}(z_J) \) reaches the far left, the vacuum expectation vanishes on $\mf{c}_1$ but not on $\mf{c}_2$. Thus, we obtain
\begin{equation}
T_2=\Big\langle\mathcal{E}_{\delta}(z_J)\prod_{i\notin I}\mathcal{E}_{\mu_i}(0	)\prod_{j\notin J}\mathcal{E}_{-\nu_j}(z_j)\Big\rangle\,. \label{equ:t2}
\end{equation}
Comparing \cref{equ:comparet2} and \cref{equ:t2}, the only difference is the operator on the far left. By a similar argument as in our computation of $T_1$, this vacuum expectation vanishes whenever the operator in \( \mathcal{E}_{\delta}(z_J) \) is not only involved in the last commutation. Thus, the last step of the algorithm for \cref{equ:t2} ends with
\begin{align}
\Big\langle\mathcal{E}_{\delta}(z_J) \mathcal{E}_{-\delta}(z_{J^c})\Big\rangle=\frac{\varsigma (\delta z_{[n]})}{\varsigma (z_{[n]})}
\end{align}
instead of the last step for \cref{equ:comparet2}, which ends with
\begin{align}
\Big\langle\mathcal{E}_{\delta}(0)\mathcal{E}_{-\delta}(z_{J^c})\Big\rangle=\frac{\varsigma (\delta z_{J^c})}{\varsigma (z_{J^c})}\,.
\end{align}
Therefore the following equality holds for $T_2$:
\begin{align}
T_2=\frac{\varsigma (z_{J^c})\varsigma (\delta z_{[n]})}{\varsigma (\delta z_{J^c})\varsigma (z_{[n]})}\bigg\langle\mathcal{E}_{\delta}(0)\prod_{i\notin I}\mathcal{E}_{\mu_i}(0)\prod_{j\notin J}\mathcal{E}_{-\nu_j}(z_j)\bigg\rangle\,.
\end{align}
Substituting $q^{\mf{c}_2}(z_1,\dots,z_n)-q^{\mf{c}_1}(z_1,\dots,z_n)=T_1T_2$ into \cref{equ:diff}, we obtain
\begin{align}
\begin{split}
\mathcal{WC}^{\le}_{\delta}=&\sum_{v_1,\dots,v_n=0}^{\infty}u_1^{v_1}\cdots u_n^{v_n}\prod_{j=1}^n\frac{(v_j+\nu_j-1)!}{(\nu_j-1)!}\prod_{j=1}^n\mathcal{S}(z_j)^{\nu_j-1} \frac{\varsigma (z_J)\varsigma (z_{J^c})\varsigma (\delta z_{[n]})}{\varsigma (\delta (z_J))\varsigma (\delta z_{J^c}) \varsigma (z_{[n]})}\\
&\frac{\left\langle\prod_{i\notin I}\mathcal{E}_{\mu_i}(0)\prod_{j\notin J}\mathcal{E}_{-\nu_j}(z_j)\mathcal{E}_{-\delta}(0)\right\rangle\left\langle\mathcal{E}_{\delta}(0)\prod_{i\notin I}\mathcal{E}_{\mu_i}(0)\prod_{j\notin J}\mathcal{E}_{-\nu_j}(z_j)\right\rangle}{\prod_{i=1}^m\mu_i\prod_{j=1}^n\nu_j}\,.
\end{split}
\end{align}
Comparing this extension to the following extension of the product $\mathcal{H}_{\mu^I,\nu^J+\delta}\mathcal{H}_{\mu^{I^c},\nu^{J^c}+\delta}$,
\begin{align}
\begin{split}
&\mathcal{H}^{\le}_{\mu^I,\nu^J+\delta}(\underline{u}^J,u',\underline{z}^J,0)\mathcal{H}^{\le}_{\mu^{I^c}+\delta,\nu^J}(\underline{u}^{J^c},\underline{z}^{J^c})=\sum_{v_1,\dots,v_n,v'=0}^{\infty}u_1^{v_1}\cdots u_n^{v_n}u'^{v'}\\
&\prod_{j=1}^n\frac{(v_j+\nu_j-1)!}{(\nu_j-1)!}\frac{(v'+\delta-1)!}{(\delta-1)!}\prod_{j=1}^n\mathcal{S}(z_j)^{\nu_j-1}\mc{S}(0)^{\delta-1}\\
&\frac{\left\langle\prod_{i\in I}\mathcal{E}_{\mu_i}(0)\prod_{j\in J}\mathcal{E}_{-\nu_j}(z_j)\mathcal{E}_{-\delta}(0)\right\rangle\left\langle\mathcal{E}_{\delta}(0)\prod_{i\notin I}\mathcal{E}_{\mu_i}(0)\prod_{j\notin J}\mathcal{E}_{-\nu_j}(z_j)\right\rangle}{\prod_{i=1}^m\mu_i\prod_{j=1}^n\nu_j\delta^2}\,,
\end{split}
\end{align}
we see immediately that 
\begin{equation}
\mathcal{WC}^{\le}_{\delta}(\underline{u},\underline{z})=\delta^2 \frac{\varsigma (z_J) \varsigma (z_{J^c}) \varsigma (\delta z_{[n]})}{\varsigma (\delta z_J) \varsigma(\delta z_{J^c})\varsigma (z_{[n]})} [u'^0]\mathcal{H}^{\le}_{\mu^I,\nu^J+\delta}(\underline{u}^J,u',\underline{z}^J,0)\mathcal{H}^{\le}_{\mu^{I^c}+\delta,\nu^J}(\underline{u}^{J^c},\underline{z}^{J^c})\,,
\end{equation}
as desired.
\end{proof}

\subsection{Wall-crossing formulae for triply mixed Hurwitz numbers}
In this subsection we deal with the triply mixed Hurwitz numbers. The procedure is very similar to one in the previous subsection, so we only outline the main steps and give the results.  We begin by defining the refined generating series for triply mixed Hurwitz numbers.

\begin{definition}
Let $\mu$ and $\nu$ be partitions as before. We define the \textit{refined triply mixed generating series} as
\begin{align}
\mathcal{H}^{(2),\le,<}_{\mu,\nu}(\underline{t},\underline{u},X,\underline{y},\underline{z}) \coloneq &\sum_{\substack{v_1,\dots,v_n=0 \\ w_1, \dots, w_n = 0 \\ 0 \leq w_i \leq \nu_i }}^{\infty}t_1^{v_1}\cdots t_n^{v_n} u_1^{w_1}\cdots u_n^{w_n}\prod_{j=1}^n\frac{(v_j+\nu_j-1)!}{(\nu_j-w_j)!}\prod_{j=1}^n\frac{\mathcal{S}(y_j)^{\nu_j-1}}{\mathcal{S}(z_j)^{\nu_j+1}}\\
&\qquad \frac{\left\langle\prod_{i=1}^{m}\mathcal{E}_{\mu_i}(0)\prod_{j=1}^n\mathcal{E}_{-\nu_j}(X\nu_j+y_j+z_j)\right\rangle}{\prod_{i=1}^m\mu_i}\,.\label{equ:refinedtrip}
\end{align}
Moreover, let $\mathfrak{c}$ be induced by the hyperplane arrangement $\mathcal{W}$ and denote by \( q^{\mathfrak{c}}(X,\underline{y},\underline{z}) \) the polynomial expressing
\begin{equation}
\bigg\langle\prod_{i=1}^{m}\mathcal{E}_{\mu_i}(0)\prod_{j=1}^n\mathcal{E}_{-\nu_j}(X\nu_j+y_j+z_j)\bigg\rangle
\end{equation}
in the chamber $\mathfrak{c}$. Then we define
\begin{equation}
\mathcal{H}^{(2),\le,<}_{\mu,\nu}(\mathfrak{c},\underline{t},\underline{u},X,\underline{y},\underline{z}) \coloneq \!\!\!\! \sum_{\substack{v_1,\dots,v_n=0 \\ w_1, \dots, w_n = 0 \\ 0 \leq w_i \leq \nu_i }}^{\infty}\!\!\!\! t_1^{v_1}\cdots t_n^{v_n} u_1^{w_1}\cdots u_n^{w_n}\prod_{j=1}^n\frac{(v_j+\nu_j-1)!}{(\nu_j-w_j)!}\prod_{j=1}^n\frac{\mathcal{S}(y_j)^{\nu_j-1}}{\mathcal{S}(z_j)^{\nu_j+1}} \frac{q^{\mathfrak{c}}(X,\underline{y},\underline{z})}{\prod_{i=1}^m\mu_i}\,.
\end{equation}
Let $\delta=\sum_{i\in I}\mu_i-\sum_{j\in J}\nu_j= 0$ define a wall in $\mathcal{W}$ and let $\mathfrak{c}_1$, $\mathfrak{c}_2$ be chambers separated by this wall. Define
\begin{equation}
\mathcal{WC}^{(2),\le,<}_{\delta}(\underline{t},\underline{u},X,\underline{y},\underline{z}) \coloneq \mathcal{H}^{(2),\le,<}_{\mu,\nu}(\mathfrak{c}_2,\underline{t},\underline{u},X,\underline{y},\underline{z})-\mathcal{H}^{(2),\le,<}_{\mu,\nu}(\mathfrak{c}_1,\underline{t},\underline{u},X,\underline{y},\underline{z})\,.
\end{equation}
\end{definition}

As in the previous subsection, we have the following lemma

\begin{lemma}
Let $g$, $p$, $q$ and $r$ be a non-negative integer, $\mu,\nu$ ordered partitions as before and let $b=2g-2+m+n=p+q+r$, then
\begin{equation}
h_{p,q,r;\mu,\nu}^{(2),\le,<}=\!\!\!\!\!\!\!\!\!\!\! \sum_{\substack{v_1,\dots,v_n=0\\w_1,\dots,w_n=0\\ |\underline{v}|=p,|\underline{w}|=q, 0 \leq w_i \leq \nu_i}}\!\!\!\!\!\!\!\!\!\!\!\!\left[X^py_1^{v_1}\cdots y_n^{v_n}z_1^{w_n}\cdots z_n^{w_n}\right] \left[t_1^{v_1}\cdots t_n^{v_n}u_1^{w_1}\cdots u_n^{w_n}\right]\mathcal{H}^{(2),\le,<}_{\mu,\nu}(\mathfrak{c},\underline{t},\underline{u},X,\underline{y},\underline{z})
\end{equation}
and for a wall $\delta$ separating $\mathfrak{c}_2$ and $\mathfrak{c}_1$, we obtain
\begin{equation}
WC^{(2),\le,<}_{\delta}= \!\!\!\!\!\!\!\!\!\!\! \sum_{\substack{v_1,\dots,v_n=0\\w_1,\dots,w_n=0\\ |\underline{v}|=p,|\underline{w}|=q, 0 \leq w_i \leq \nu_i}}\!\!\!\!\!\!\!\!\!\!\!\! \left[X^py_1^{v_1}\cdots y_n^{v_n}z_1^{w_n}\cdots z_n^{w_n}\right] \left[t_1^{v_1}\cdots t_n^{v_n}u_1^{w_1}\cdots u_n^{w_n}\right]\mathcal{WC}^{(2),\le,<}_{\delta}(\underline{t},\underline{u},X,\underline{y},\underline{z})\,.
\end{equation}
\end{lemma}

By a similar calculation as in the proof of theorem \ref{theorem:wallcrossingmon}, we get the following result.

\begin{theorem}
Let $\mu,\nu$ be ordered partitions of the same positive integer and let $\delta=\sum_{i\in I}\mu_i-\sum_{j\in J}\nu_j$. Then
\begin{align}
\mathcal{WC}^{(2),\le,<}_{\delta}(\underline{t},\underline{u},X,\underline{y},\underline{z})&=
\delta^2\frac{\varsigma (A_J+X\delta )\varsigma (A_{J^c})\varsigma (\delta A_{[n]})}{\varsigma \big(\delta(A_J+X\delta )\big)\varsigma (\delta A_{J^c} )\varsigma (A_{[n]} )}\\
&\qquad \left[(t')^0 (u')^0 \right] \mathcal{H}^{(2),\le,<}_{\mu^I,\nu^J+\delta}(\underline{t}^J,t',\underline{u}^J,u',X,\underline{y}^J,0,\underline{z}^J,0)\\
&\qquad \mathcal{H}^{(2),\le,<}_{\mu_{I^c}+\delta,\nu_{J^c}}(\underline{t}^{J^c},\underline{u}^{J^c},X,\underline{y}^{J^c},\underline{z}^{J^c})\,,
\end{align}
where
\begin{equation}
A_J=\sum_{j\in J}X\nu_j+y_j+z_j
\end{equation}
and the zero arguments in the first \( \mc{H} \) are the \( y \) and \( z \) variables corresponding to the part \( \delta \) of the partition \( \nu^J + \delta \).
\end{theorem}
So also in the general, mixed, case, the wall-crossing generating function can be related to a product of two Hurwitz generating functions of lower degree.

%%%%%%%%%%%%%%%%%%%%%%%%%%%%%%%%
%%%                               SECTION
%%%%%%%%%%%%%%%%%%%%%%%%%%%%%%%%

\section{Hypergeometric tau functions}\label{HypGeomTau}
Let us consider the family of 2D-Toda hypergeometric tau functions $\tau_{(q,w,z)}(\mathbf{t},\mathbf{\tilde{t}})$ in the sense of Harnad and Orlov \cite{HO2} (for $N=0$). They are defined as
\begin{equation}\label{eq:HarnadO}
\tau_{(q,w,z)}(\mathbf{t},\mathbf{\tilde{t}}) := 
\sum_{n=0}q^n \sum_{\lambda\vdash n} \prod_{j=1}^n \frac{\prod_{a=1}^l (1 + \mathsf{cr}^{\lambda}_j w_a)}{\prod_{b=1}^m (1 - \mathsf{cr}^{\lambda}_j z_b)} s_\lambda(\mathbf{t}) s_\lambda(\mathbf{\tilde{t}})
\end{equation}
where the variables $q$, $w_a$ and $z_b$ are the parameters of the tau-function. After expanding the Schur function in terms of the power sums, the coefficient of 	
$
q^n \prod_{a=1}^l \prod_{b=1}^m  w_a^{c_a}z_b^{d_b} p_\mu(\mathbf{t}) p_\nu(\mathbf{\tilde t})
$
can be expressed in terms of operators acting on the Fock space as 
\begin{equation}\label{eq:coeff2TODA}
\frac{1}{\prod_i \mu_i \prod_j \nu_j}\Big[ \prod_{a=1}^l w_a^{c_a} \prod_{b=1}^m z_b^{d_b}\Big]. \bigg\< \prod_{i=1}^{\ell(\nu)} \alpha_{-\nu_j}
\prod_{a=1}^l \mathcal{D}^{(\sigma)}(w_a) \prod_{b=1}^m  \mathcal{D}^{(h)}(z_b)
\prod_{i=1}^{\ell(\mu)} \alpha_{-\mu_i} \bigg \>,
\end{equation}
which corresponds to the mixed monotone/strictly monotone case, where an arbitrary finite number of operators of each type is allowed. Again, we insert a trivial factor
$$
1 = \prod_{b=m}^1  \mathcal{D}^{(h)}(z_b)^{-1}
\prod_{a=l}^1 \mathcal{D}^{(\sigma)}(w_a)^{-1} \prod_{a=1}^l \mathcal{D}^{(\sigma)}(w_a) \prod_{b=1}^m  \mathcal{D}^{(h)}(z_b)
$$
between each $\alpha_{-\mu_i}$ and $\alpha_{-\mu_{i+1}}$, for $i = 1, \dots, \ell(\mu) - 1$. We moreover insert the operator
$$
\prod_{b=m}^1  \mathcal{D}^{(h)}(z_b)^{-1}
\prod_{a=l}^1 \mathcal{D}^{(\sigma)}(w_a)^{-1}
$$
between $\alpha_{-\mu_{\ell(\mu)}}$ and the vacuum. Again, note that the insertion does not modify the expression, since the operator is of exponential form $e^B$ where $B$ is an operator annihilating the vacuum. At this point, we are ready to compute $a+b$ nested conjugations by applying lemma \ref{lem:ConjDE} $a + b$ times to each expression
$$
\prod_{a=1}^l \mathcal{D}^{(\sigma)}(w_a) \prod_{b=1}^m  \mathcal{D}^{(h)}(z_b) \;
\mathcal{E}_{-\mu_i}(A = 0) \; \prod_{b=m}^1  \mathcal{D}^{(h)}(z_b)^{-1} \prod_{a=l}^1 \mathcal{D}^{(\sigma)}(w_a)^{-1}, \qquad i = 1, \dots, \ell(\mu)
$$
obtaining for each $i = 1, \dots, \ell(\mu)$:
\begin{multline}
\sum_{\substack{v_{1,i}, \dots, v_{m,i} = 0 \\ t_{1,i}, \dots, t_{l,i} = 0}} \prod_{b=1}^m \frac{(v_{b,i} + \mu_i - 1)!}{(v_{b,i} - 1)!} \prod_{a=1}^l \frac{\mu_i!}{(\mu_i - t_{a,i})!} \times
\\
\times[x_{1,i}^{v_{1,i}} \cdots x_{m,i}^{v_{m,i}} y_{1,i}^{t_{1,i}} \dots y_{l,i}^{t_{l,i}}].
\frac{\prod_{b=1}^m \mathcal{S}(z_b x_{b,i})^{\mu_i -1}}{\prod_{a=1}^l \mathcal{S}(w_a y_{a,i})^{\mu_i +1}} \mathcal{E}_{-{\mu_i}} \left( \sum_{b=1}^m z_bx_{b,i} + \sum_{a=1}^l w_ay_{a,i}\right).
\end{multline}
The expression for the coefficients in \eqref{eq:coeff2TODA} then reads
\begin{multline}
\sum_{\substack{v_{b,i}, t_{a,i} =\, 0; \\ i = 1, \dots, \ell(\mu); \\
b = 1, \dots, m; \, a = 1, \dots, l;
\\
\sum_i v_{b,i} = d_b, \sum_i t_{a,i} = c_a}}^{\text{finite}}
\!\!\!\!\! \prod_{i=1}^{\ell(\mu)} \prod_{b=1}^m \frac{(v_{b,i} + \mu_i - 1)!}{(v_{b,i} - 1)!} \prod_{a=1}^l \frac{\mu_i!}{(\mu_i - t_{a,i})!} [ \prod_{i=1}^{\ell(\mu)} x_{1,i}^{v_{1,i}} \cdots x_{m,i}^{v_{m,i}} y_{1,i}^{t_{1,i}} \dots y_{l,i}^{t_{l,i}}].
\\
\frac{1}{\prod_i \mu_i \prod_j \nu_j} \prod_{i=1}^{\ell(\mu)} \frac{\prod_{b=1}^m \mathcal{S}(z_b x_{b,i})^{\mu_i -1}}{\prod_{a=1}^l \mathcal{S}(w_a y_{a,i})^{\mu_i +1}}
\bigg \< \prod_{j=1}^{\ell(\nu)}\mathcal{E}_{\nu_j}(0)\prod_{i=1}^{\ell(\mu)}\mathcal{E}_{-{\mu_i}} \left( \sum_{b=1}^m z_bx_{b,i} + \sum_{a=1}^l w_ay_{a,i}\right) \bigg \>
\end{multline}
Let us now make the following observations:
\begin{enumerate}
\item \textit{The vacuum expectation.} By the adapted version of Johnson's algorithm in section \ref{sec:Johnson}, the vacuum expectation is equal to a finite sum of finite products of $\varsigma$ functions whose arguments are linear combinations of the variables $\mu_jz_bx_{b,i}, \nu_jz_bx_{b,i}, \mu_jw_ay_{a,i},$ and $\nu_jw_ay_{a,i}$, times one single extra $\varsigma$ function at the denominator, whose argument is given by the sum of all the variables above. Recall that 
$\varsigma(Z) = Z + O(Z^3)$
is an (odd) analytic function, therefore no poles are produced by the $\varsigma$ at the numerator, and the conditions
$
\sum_i v_{b,i} = d_b, \sum_i t_{a,i} = c_a
$
ensure boundedness in the degree, and therefore polynomiality. Again, the only possible pole coming from the function
$
\frac{1}{\varsigma(Z)} = \frac{1}{Z} + O(Z)
$
at the denominator, where here $Z$ is the formal sum of all the four types of variables above, is removable since it simplifies against the argument of the $\varsigma$ function produced by the last commutation of each commutation pattern.

\item \textit{The ratio of products of $\mathcal{S}$ functions.} Recall that both $\mathcal{S}(Z)$ and $\mathcal{S}(Z)^{-1}$ are analytic functions in $Z$, and so are their positive powers. Again, the conditions
$
\sum_i v_{b,i} = d_b, \sum_i t_{a,i} = c_a
$
ensure boundedness in the degree.
\item \textit{The product of ratio of factorials.} Each ratio of factorials of the form
$
\frac{(v_{b,i} + \mu_i - 1)!}{(v_{b,i} - 1)!}, 
$
or
$
 \frac{\mu_i!}{(\mu_i - t_{a,i})!},
$
is a polynomial in $\mu_i$ of degree $v_{b,i}$ or $t_{a,i}$, respectively. Once more, the conditions
$
\sum_i v_{b,i} = d_b, \sum_i t_{a,i} = c_a
$
ensure boundedness in the degree.
\item \textit{Possible simple poles in the zero parts}. By the finiteness of the first sum, we are only left with checking that the simple poles coming from the factor $(\prod_i \mu_i \prod_j \nu_j)^{-1}$ are removable. This check is totally analogous to the proof of theorem \ref{thm:piecewise}. Simplifying $\nu_j^{-1}$ is easy: the first commutation relation for $\mathcal{E}_{\nu_j}(0)$ with whatever $\mathcal{E}$ operator is determined by the commutation pattern reads
$[\mathcal{E}_{\nu_j}(0), \mathcal{E}_A(W)] = \varsigma(\nu_j W)\mathcal{E}_{\nu_j + A}(W),$
which is divisible by $\nu_j$ (even in case this commutation is the very last of the commutation pattern, we have
$\big \< [\mathcal{E}_{\nu_j}(0), \mathcal{E}_{-{\nu_j}}(W)]\big \> = \frac{\varsigma(\nu_j W)}{\varsigma(W)},$
which is still divisible by $\nu_j$ after the removal of the simple pole in $W=0$). Simplifying the factor $\mu_i^{-1}$ is also similar to previous cases: note that the ratio of factorials $\frac{(v_{b,i} + \mu_i - 1)!}{(v_{b,i} - 1)!}$, or $ \frac{\mu_i!}{(\mu_i - t_{a,i})!} $ are divisible by $\mu_i$, unless $v_{b,i}$ or $t_{a,i}$ are zero, respectively. Therefore we are only left with checking the summands in which $v_{b,i} = t_{a,i} = 0$ for all $b = 1, \dots, m$ and for all $a = 1, \dots, l$ for a fixed index $i$ (this does \textit{not} imply $c_a = d_b = 0$). However, in these summands the $i$-th operator $\mathcal{E}_{-{\mu_i}} \left( \sum_{b=1}^m z_bx_{b,i} + \sum_{a=1}^l w_ay_{a,i}\right)$ becomes $\mathcal{E}_{-{\mu_i}} ( 0)$, and therefore produces a factor of the form $\varsigma(-\mu_i W)$ at its earliest commutation, which again is divisible by $\mu_i$, even in case the commutation is the very last one.
\end{enumerate}

This immediately leads to the following proposition.
\begin{proposition}
The coefficients of the 2D-Toda hypergeometric tau-functions of equation \eqref{eq:HarnadO} in the basis of power sums are piecewise polynomial in the parts of the partitions indexing the power sums, and those polynomials can be explicitly computed via the algorithm described in section \ref{sec:Johnson}.
\end{proposition}

\printbibliography

\end{document}